\documentclass[12pt]{amsart}
\usepackage{latexsym,fancyhdr,amssymb,color,amsmath,amsthm,graphicx,listings,comment}
\usepackage[section]{placeins}
\newtheorem{thm}{Theorem} 
 
\newtheorem{definition}{Definition} 
\newtheorem{conj}{Conjecture} 
\newtheorem{coro}{Corollary}
\let\paragraph\subsection

\title{Density of wave fronts}
\author{Emily Kang and Oliver Knill}
\date{Jan 26, 2025, Aug 23, 2025, Dec 22, 2025, Jan 12, revised }
\address{Department of Mathematics \\ Harvard University \\ Cambridge, MA, 02138 }
\subjclass{}

\keywords{Wave fronts, Geodesics, Billiards}

\begin{document}
\maketitle

\begin{abstract}
We prove that wave fronts on a flat torus become dense.
As a corollary, wave fronts become dense for a square billiard and for the geodesic flow on the
flat Klein bottle and for the surface of a three dimensional cube.
\end{abstract} 

\section{Motivating questions}

\paragraph{}
Imagine a square room with mirror lining the inner walls. Light
emerges from a lamp. The wave front initially forms a perfect circle.
After reflecting at the mirror walls, it will bounce back, continuing
to be piecewise smooth curve at any time. What happens in the long term
with this wave front? Does it fill out the room more and more densely? 
How long do we have to wait so that every ball of a given fixed radius has
and remains to have a non-empty intersection with this curve?

\paragraph{}
An earth quake has its epicenter at a point on a closed surface. 
What is the fate of the resulting tsunami wave front emanating from this point?
On a round sphere or round projective plane, this wave will stay circular and 
oscillate forth and back periodically. This situation is obviously very special 
on a perfect sphere. It changes if the metric on the sphere is altered in any way.
What happens for a general surface like an ellipsoid or a torus, surfaces that lack
this return property? What happens with a light front radiating from a point 
on the surface of a cube? 

\paragraph{}
As a third example, imagine a duck in a circular pond, making waves. These waves
first form circles around the duck, then bounce off the shore and come back. 
What happens in the long term? If the duck was in the center of the pond, and assuming
ideal conditions, the wave front would reach the shore, reflect there and return back to the duck. 
Assuming no friction or diffusion, the wave front would oscillates back and forth
between center and boundary. What would happen if the duck was located at a point off the center point of the pond? 
Would the wave front produced by the duck become dense on the pond's surface? Intuition says yes
because the length of the wave front grows \cite{Vicente2020}. It gets more complicated as
it has to remain wrapped-up in the pond.

\paragraph{}
These pictures illustrate some rather general {\bf wave front problems} 
in {\bf Riemannian geometry}. One can ask for which Riemannian manifolds 
(with or without boundary or on which polytopes) do the wave front become dense 
for all initial points, like in the mirrored room example?
In which cases does the wave front become dense for some points and not become dense for others
like in the duck example case? In which cases does it not become dense for all initial points,
like in the tsunami case? 

\paragraph{}
We explore this theme here in the torus, square and cube surface case, where the metric is flat
and where the question can be analyzed with elementary tools only. While the subject could reach 
far into differential geometry, all we really need here is some single variable calculus. 

\begin{figure}[!htpb]
\scalebox{0.64}{\includegraphics{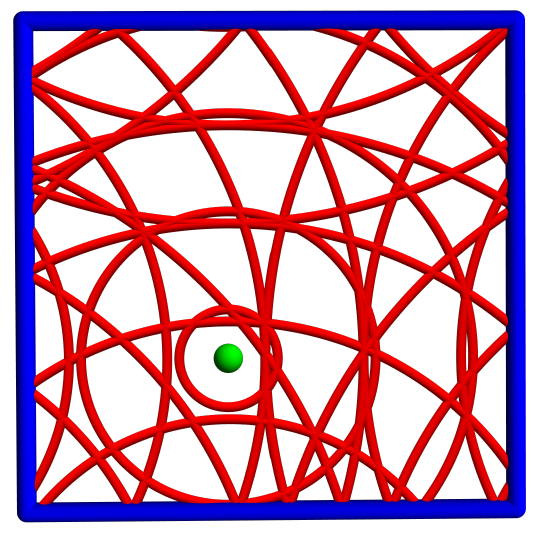}}
\scalebox{0.64}{\includegraphics{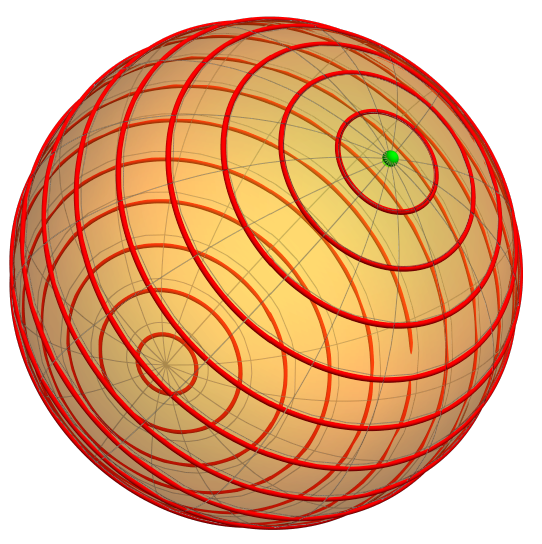}}
\scalebox{0.28}{\includegraphics{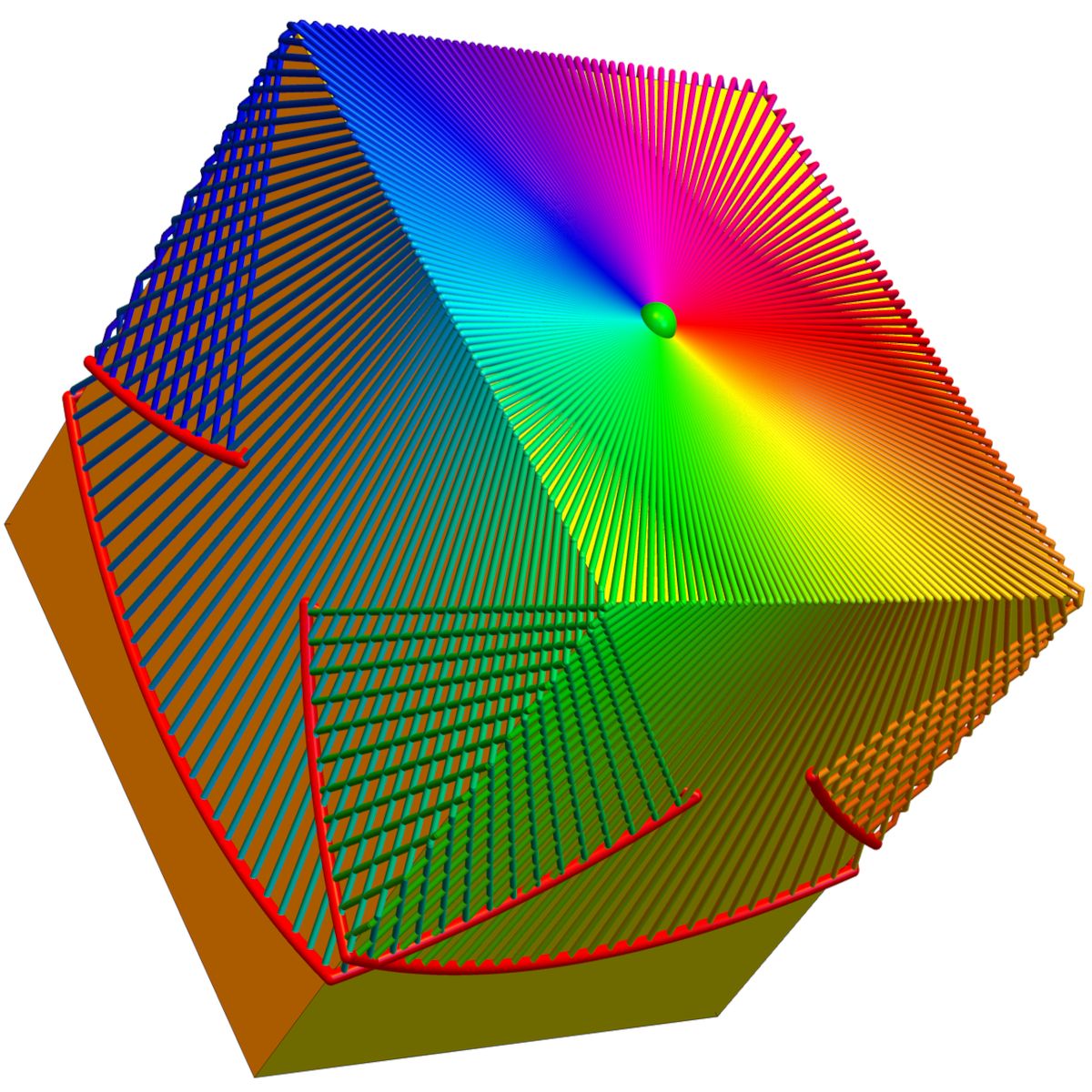}}
\scalebox{0.64}{\includegraphics{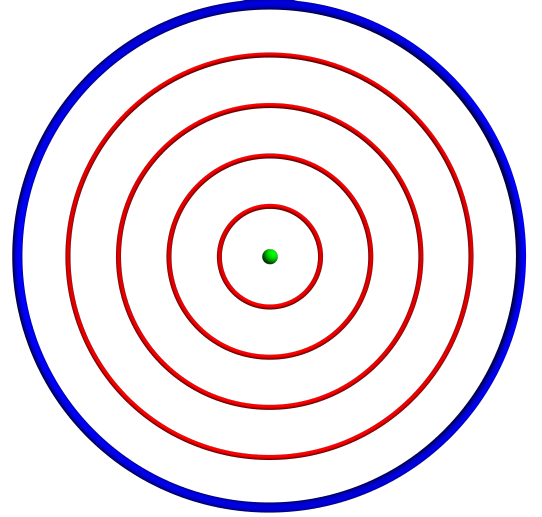}}
\label{Motivation}
\caption{
What happens with a light wave front in a room of mirrors, a 
tsunami on a perfect sphere or a cube, and the wave front produced
by a duck in a circular pond?
}
\end{figure}  

\section{Background}

\paragraph{}
We begin with reviewing some elementary notions in differential geometry
as covered in standard textbooks like \cite{AMR, BergerGostiaux, DoCarmo1992,Kuehnel2015}.
A good overview of the subject is Chapter 10 in \cite{BergerPanorama}.
For manifolds with boundaries, the geodesic dynamics, often referred to as a
{\bf billiard}, has become a research topic on its own right 
\cite{KozlovTreshchev,KatokStrelcyn,Tab95,ChernovMarkarian}.
A reference for wave fronts in Riemannian manifolds with boundary is \cite{AlexanderBergBishop}.
Motivated by the questions posed in the previous section we work in three different frameworks. 

\paragraph{}
The first setting involves a compact Riemannian manifold $(M,g)$, where $M$ is a smooth manifold equipped
with a positive definite symmetric $(2,0)$ tensor field $g$. At each point $P \in M$, the bilinear form 
$g(P)$ defines an inner product $\langle v,w \rangle_P = g(P)(v,w)$ in the tangent space $T_PM$.                    
This inner product introduces a notion of length $||v||_p = \sqrt{\langle v,v \rangle}_p$,
which, in turn, allows us to define the {\bf arc length} of a smooth curve      
$x:[a,b] \to M$ as $\int_a^b ||\dot{x}(t)|| \; dt$, where $\dot{x} \in T_xM$ is the velocity vector.
By the fundamental theorem of Riemannian geometry, there is a unique torsion-free connection
on $(M,g)$. This covariant derivative is in a local basis $\{ e_1, \dots,e_d \}$ at $P$ determined by the
Christoffel symbols $\nabla_{e_i} e_j = \sum_k \Gamma_{ij}^k e_k$. The geodesic flow can then be defined
as the solution to the differential equation 
$\ddot{x}^k + \sum_{i,j=1}^d \Gamma_{ij}^k(x(t)) \dot{x}^i \dot{x}^j = 0$.
Given an initial point $P=x(0)$ and an initial velocity $\dot{x}(0)$, there is a unique solution curve
$x(t)$ to these equations. 
If $M$ is compact, it is also a complete metric space and the Hopf-Rynov theorem assures      
global existence of solutions.

\paragraph{}
The second framework extends to compact Riemannian manifold $(M,g)$ with boundary $\delta M$. Here,
the geodesic flow, referred to as a {\bf billiard}, requires distinguishing 
between the motion in the interior of $M$ and reflections at the boundary.                                
The boundary $N=\delta M$ is assumed to be a smooth manifold of co-dimension $1$ with the metric $g$
inducing a Riemannian metric $h$ on $\delta M$.
This is achieved by restricting $g$ to the hyper tangent plane $T_PN \subset T_PM$. 
If a geodesic $x(t)$ starts at the boundary, it can continue as a boundary geodesic or enter the 
interior of $M$. A geodesic $x(t)$ starting in the interior hits the boundary at some point           
$P \in \delta M$, the incoming  vector $x'(t^-) \in T_PM$ 
is reflected at the boundary to produce an out-going vector $x'(t^+)$, analogous
to a billiard ball rebounding off a table edge or a light ray reflecting off a mirror. 
In non-convex scenarios, like the Sinai billiard, where a circular obstacle is placed in 
a flat Clifford 2-torus, the return map to the boundary may already
exhibit non-smooth behavior. After each reflection, the geodesic flow resumes in the                                        
interior until the next impact point. If a geodesics hits a corner, then it is terminated.

\paragraph{}
The third framework considers {\bf piecewise smooth Riemannian polytops},
a concept that was historically challenging to define \cite{Richeson,lakatos}.
This difficulty has led either to restricting the analysis to convex polytopes 
\cite{gruenbaum,Ziegler} or then to use a broader definition, where a polytop $M$ is 
viewed as the geometric realization of a finite abstract simplicial complex $G$, 
a finite set of non-empty sets closed under the operation
of taking non-empty subsets \cite{DehnHeegaard}.
A {\bf piecewise smooth Riemannian polytop} is such a polytop that is equipped
with a Riemannian metric $g$ that is smooth except on a subset of co-dimension $2$.                           
Inductively, this means $M$ is a smooth manifold (with or                  
without boundary) for which the tensor $g$ is defined on an open subset
$U = M \setminus K$, where $K$ is a piecewise smooth Riemannian polytope of co-dimension $2$.
The geodesic flow is defined for geodesic paths avoiding $K$. One of the simplest examples is
the square table, a $2$-manifold with circular boundary and a metric that is discontinuous 
at its four corners. For a $3$-dimensional example, consider a billiard in a solid cube.
Here, the Riemannian metric tensor $g$ may be discontinuous along the edges and vertices that
form a 1-dimensional network. 

\section{Wave fronts} 

\paragraph{}
{\bf Wave fronts} can be defined and observed on any {\bf piecewise smooth 
Riemannian manifold $M$ with or without boundary}. 
The {\bf wave front} of a point $P$ is defined as the hyper surface $W_t = \exp_P(S_t)$,
where $\exp_P: T_PM \to M$ is the {\bf exponential map} and $S_t=\{ x \in \mathbb{R}^d, |x|=t\}$ 
is the sphere of radius $t$ in the tangent space $T_PM \sim \mathbb{R}^d$. The global existence
of the exponential map on a closed smooth Riemannian manifold is guaranteed by the Hopf-Rynov theorem. 
In the case of a manifold with boundary, the geodesic flow dynamics is also known as the
{\bf billiard dynamical system}. This is already interesting for a 2-dimensional domain and especially 
on a convex domains, in which case the system is called the {\bf Birkhoff billiard} \cite{Birkhoff,Por50}.  
Already in a polygon, like in a square, a path entering one of the corners does not have a well
defined bounce, but the wave front $W_t$ remains defined for all times.
Similarly, when looking at a wave front on a surface of a polyhedron like a cube, it is not clear how
to continue the geodesic through the corner. We just disregard these trajectories. The wave front
$W_t$ is still defined now as piecewise smooth curve. It will even become disconnected after
finite $t$ in general. In this article, we limit ourselves to
\begin{itemize}
\item 1. Smooth closed manifolds. In this case, $W_t$ always are continuous curves, which are piecewise 
smooth with singularities at caustic points, points where the Jacobian $d \exp_P$ is singular. 
\item 2. Convex planar flat billiards. Now, $W_t$ remains a continuous curve. It remains piecewise smooth. 
Differentiability can fail at caustics or at the boundary.
\item 3. Piecewise flat convex polyhedra. In this case, $W_t$ in general can already 
become disconnected.
\end{itemize}

\begin{figure}[!htpb]
\scalebox{0.15}{\includegraphics{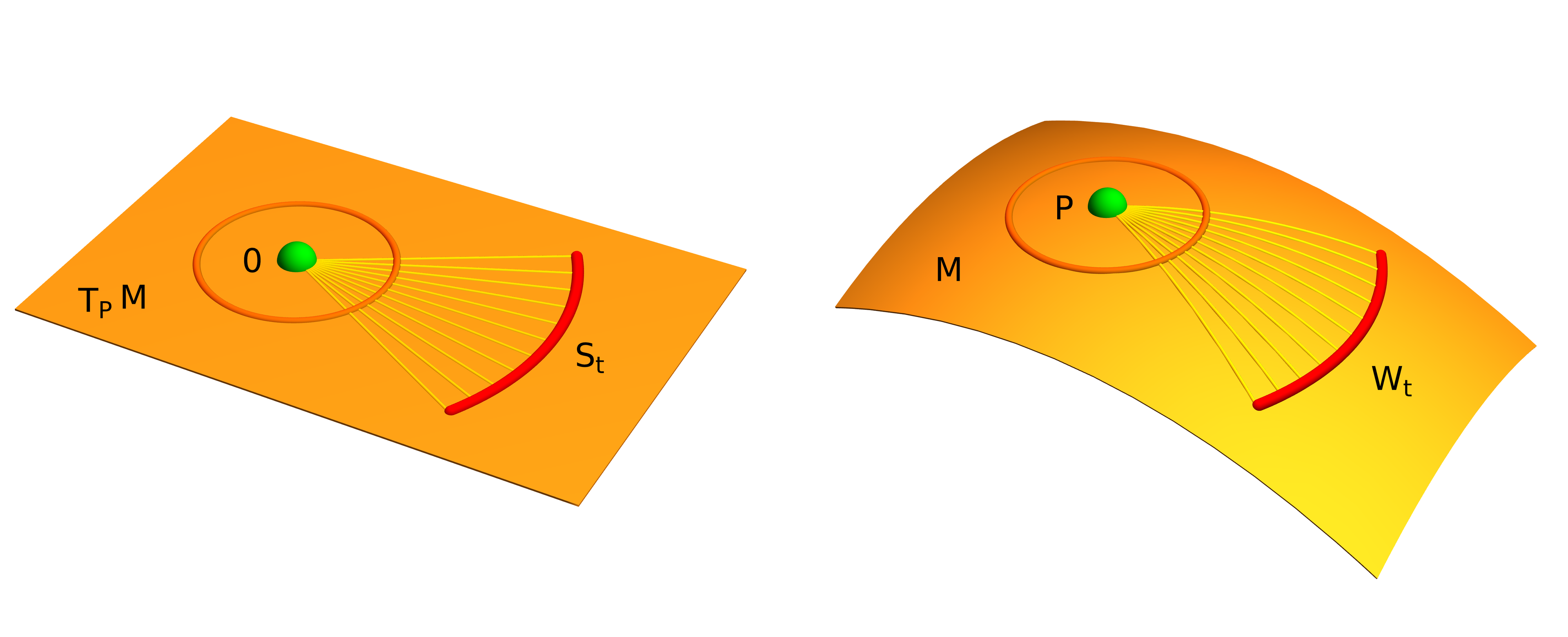}}
\label{Exponential map}
\caption{
In a Riemannian d-manifold $M$, the wave front $W_t(P)$ of a point $P$ for which a tangent space
$T_PM$ exists, is the image $\exp_P(S_t)$ of the geodesic sphere 
$S_t=\{ |x|=t\} \subset T_PM \sim \mathbb{R}^d$ of radius $t$.
}
\end{figure}

\paragraph{}
While for a smooth surface $M$ without boundary or a billiard in a convex region,
the wave fronts $W_t$ are piecewise smooth and in particular continuous. This is already different
on a closed polyhedral surface. The wave front there can become disconnected 
after some time. This happens in particular on a  cube's surface, specifically the surface bounding 
the unit cube. Numerical experiments still indicate to us that also for more general non-flat 
compact surfaces, in general, the wave front becomes dense, meaning that it fills
out the surface more and more. An unresolved case is already the dodecahedron. 

\paragraph{}
Here is the definition of ``dense wave front" 
for a general compact Riemannian manifold $M$ with or without boundary.
Given $r>0$, the {\bf open ball} $B_r(Q)$ of a point $Q \in M$ is the set of points $P$ for 
which there is a geodesic from $P$ to $Q$ that has geodesic length smaller than $r$. 

\begin{definition}[Dense Wave Front]
\label{dense}
We say that the wave front $W_t(P)$ {\bf becomes dense} on $M$ if for every open 
ball $U = B_r(Q)$ in $M$ with $r>0$, there is a time 
$\tau(P,r)$ such that for all $t>\tau(P,r)$, the wave front 
$W_t(P)$ intersects $U$ in a non-empty set.
\end{definition}

\paragraph{}
To rephrase this, we can also say that a wave front $W_t(P)$ becomes dense if for every 
$\epsilon>0$ there is a time $\tau(P,\epsilon)$ such that for $t>\tau(P,\epsilon)$ the
$\epsilon$-neighborhood of $W_t(P)$ covers the entire manifold. 

\paragraph{} Since the $(d-1)$-dimensional volume of the co-dimension-one wave front $W_t$ is
expected to grow and the manifold $M$ is finite, we expect it to fill up
$M$ more and more. The term ``generic" here is used in an informal way,
but it could be made precise using the notion of {\bf Baire generic}. The set of 
Riemannian manifolds of a fixed dimension naturally can be made into a metric 
space by Nash embedding them into a large Euclidean space $E$. A metric on 
Riemannian manifolds can then be defined by taking the smallest
possible Hausdorff distance between any embeddings $M_1,M_2$ in $E$. Alternatively
one could fix a topological manifold $M$ and define a distance between 
Riemannian metrics on $M$, using this topology on metrics to define Baire genericity.

\begin{conj}
For a generic compact $d \geq 2$-dimensional Riemannian manifold with or without boundary,
the wave front $W_t(P)$ becomes dense for all points $P \in M$. 
\end{conj}

\paragraph{}
For manifolds $M$ for which all geodesics are closed \cite{Besse} like the round sphere (or in a generalized sense 
the projective space), the wave front $W_t(P)$ is either a point or a circle 
for all $t \geq 0$. Any round sphere
or projective space of constant curvature in any dimension has this property.
But this is an exceptional case as already Blaschke realized in 1921 \cite{blaschke}.
It is natural to conjecture that for a generic $2$-manifold $M$ with or
without boundary, every wave front $W_t(P)$ {\bf becomes dense} in the sense 
of Definition~\ref{dense}.

\section{Surfaces} 

\paragraph{}
We confirm here that one can have dense wave fronts even in very simple
situations, despite the fact that having a dense wave front displays some sort of mixing
property. This can happen even in the case when the geodesic flow is integrable. 
We look first at the {\bf flat 2-torus} $M=\mathbb{R}^2/\mathbb{Z}^2$. 
It is a $2$-dimensional compact Riemannian manifold with Euclidean flat metric, 
induced from the plane. 

\begin{thm}
On a flat 2-torus $M = \mathbb{R}^2/\mathbb{Z}^2$, all wave fronts
$W_t(P)$ become dense. 
\end{thm}

\paragraph{}
We prove a stronger result which states that for a specific arc
$A=\{ (\cos(\theta), \sin(\theta)), \theta \in [\alpha,\beta] \}$ of positive length, the wave 
front $W_t(P,A) = \exp_P(t A)$ becomes dense. Note that the length of $A$ goes to $0$.
We can think of the exponential map   
restricted to a sector as building a {\bf light cone} of a {\bf headlight}.
If $\alpha=\beta$, we have a {\bf laser}, a wave front going in a single direction. In other
words this is then a single geodesic ray. While for irrational $\alpha/(2\pi)$ such 
a {\bf laser beam} $\{ W_s(P), s \in [0,t] \}$ becomes dense, the {\bf wave front of this 
laser} is just the point $W_t(P)$ for all $t$. 

\paragraph{}
In the universal cover $\mathbb{R}^2$ of $\mathbb{T}^2$, the wave front is a 
round circle of radius $t$. We can assume without loss of generality that the
point $P$ is at the origin. The wave front can be seen as a union of two half circles,
both of which are graphs. The first is the graph of $f(x) = \sqrt{t^2-x^2}$. The second 
is the graph of $g(x)=-\sqrt{t^2-x^2}$. 

To prove the result, we concentrate only on a 
small rectangle $R$ containing part of the graph of $f$. It is chosen to be so wide 
that its height is approximately $1$.  
For given $t>0$, define the region $R=[a,b] \times [f(a),f(b)]$ with
$[a,b]=[-2 \sqrt{t},-\sqrt{2t}]$. The graph of $W_t(P)$ intersected with 
$R$ crosses $R$ monotonically from $(a,f(a))$ to $(b,f(b))$. Since the graph of $f$
is concave down, the derivative of $f$ is in the interval $[\sqrt{2}/\sqrt{t-2},2/\sqrt{t-4}]$
which is for $t>36/5$ smaller than $3/\sqrt{t}$. 
This implies that $|f(u+1)-f(u)| \leq 3/\sqrt{t}$ for $u \in I$.  \\

If $R$ is projected to $[0,1) \times [0,1)$ using
$\pi(x,y) = (x \; {\rm mod} 1, y \; {\rm mod} 1)$, 
then for every $P=(x,y) \in [0,1) \times [0,1)$, there exists $u \in [a,b]$ such that
$||\pi(u,\sqrt{t^2-u^2}) - (x,y)|| \leq 3/\sqrt{t}$. 
Proof: Look at all the $u \in [a,b]$ such that $u \; {\rm mod} \; 1 = x$. 
As $|f(u+1)-f(u|) \leq 3/\sqrt{t}$, one of the values $u \in \pi^{-1}(x)$  
must satisfy $|f(u)-y| \leq 3/\sqrt{t}$ so that 
$||\pi(u,\sqrt{t^2-u^2}) -(x,y)|| = \sqrt{0^2 + |f(u)-y|^2} \leq 3/\sqrt{t}$. 

We have shown that even a tiny part $W_t(P) \cap R$ of the entire wave front has become 
$3/\sqrt{t}$ dense. So, also the entire wave front is $3/\sqrt{t}$ dense. 

\begin{figure}[!htpb]
\scalebox{0.63}{\includegraphics{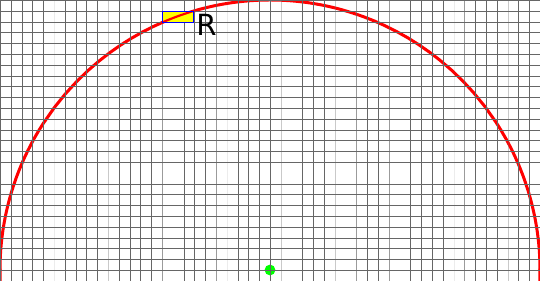}}
\scalebox{0.66}{\includegraphics{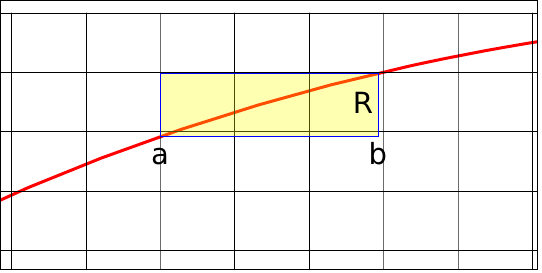}}
\scalebox{0.33}{\includegraphics{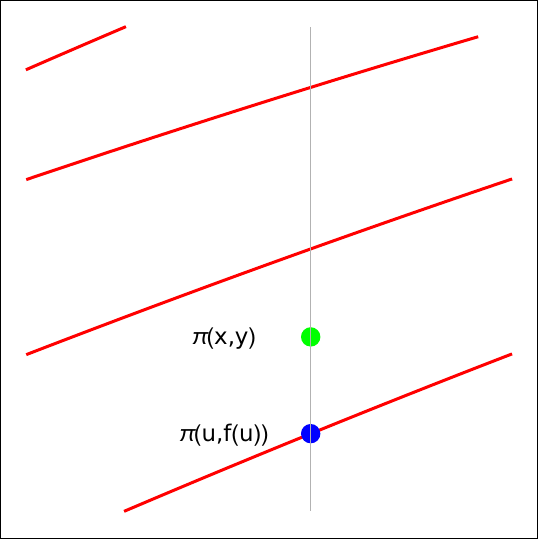}}
\label{Proof}
\caption{
To prove the result, we project a small part $R=[a,b] \times [f(a),f(b)]$ 
to the torus. This restriction to a smaller interval at first seems to make
the problem harder. But it allows to isolate part $W_t \cap R$ of the wave front $W_t$
for which $W_t$ is a graph that is asymptotically linear for $t \to \infty$.
It is then easier to see that $(W_t \cap R)/\mathbb{Z}^2$ becomes dense. 
}
\end{figure}

\paragraph{}
The flat torus $S_1 \times S_1$ be made more asymmetric by using 
$S_r \times S_s$ where $S_r,S_s$ are circles of radius $r$ or $s$. 
This could be realized by a conformal distortion of the Riemannian metric. 
It still is flat. We just need to scale the region and scale the maps accordingly 
in the proof to see that also on such asymmetric tori, the wave fronts become dense. 
A similar argument works for the {\bf flat Klein bottle} which has a flat 2-torus 
as a 2:1 cover. Both the flat Klein bottle as well as the flat 2-torus (called
Clifford torus) could be embedded isometrically in $\mathbb{R}^4$ but not in 
$\mathbb{R}^3$. 

\begin{thm}
On a flat Klein bottle $M$, all the wave fronts $W_t(P)$ become dense.
\end{thm}

\paragraph{}
We first investigated the torus problem using results from the 
{\bf Gauss-circle problem},which tells us how many lattice points we expect
in the annulus $A_{t,t+h}$ between $W_{t+h}$ and $W_t$ in the universal cover. Already, Gauss
estimated the number of lattice points in this annulus to be
close to its area $2\pi t h$ and estimated an error $|E(t)| \leq \sqrt{2} \cdot 2\pi t$.
It is known, for example, that $E(t) \leq C t^{2/3}$. 
\footnote{The big open problem in this context is whether $E(t)=O(t^{1/2+\epsilon})$ holds
for every $\epsilon>0$.}
This means that the annulus $A_{t,t+h}$ has about $2\pi t h + C t^{2/3} h$ lattice points.
For $h=t^{-1/2}$, this area grows like $2\pi \sqrt{t}$ with an error $O(t^{1/6})$.
As we expect so many lattice points in $A_{t,h}$ and every of these lattice points 
is $h=1/\sqrt{t}$ close to the wave front $W_t$ we see that $W_t$ gets $1/\sqrt{t}$
close to the original point $P$. Changing the annulus to $A_{t+h,t+2h}$ shows
that we have points in distance $[1/\sqrt{t},2/\sqrt{t}]$ from $P$.
The {\bf geometry of numbers} approach could even shed better light on {\bf how fast}
the wave fronts fill out the torus. The proof using the Gauss-circle problem in higher
dimensions would need heavier mathematical machinery like {\bf Ehrhard theory}.

\paragraph{}
The simple constructive proof given above was found only later in the summer of 2024. 
It illustrates that sometimes, it is easier to prove a stronger result: the proof 
shows that already a small part of the wave front becomes dense. The reader might 
enjoy to modify the simple single variable calculus argument to show that also on a 
stretched torus $\mathbb{R}/(\alpha \mathbb{Z}) \times \mathbb{R}/(\beta \mathbb{Z})$,
the wave front becomes dense.  
The argument can also be adapted to wave fronts on higher dimensional flat tori
$\mathbb{R}^d/\mathbb{Z}^d$. Again, we just have to restrict to a small cuboid domain $R$,
where the wave front can be described in the cover as a graph of a function 
$f(x_1, \dots, x_{d-1})$ of $d-1$ variables. There is no doubt that in the case of a 
non-flat metric, the density of the wave front persists. However this has not been proven; we
do not know the shape of the wave front in the universal cover. But it is likely 
as the wave front gets longer. In the case when the geodesic flow has positive Lyapunov 
exponents, the wave front length in the universal cover would grow exponentially and make
a density result more likely. But we even don't know how to prove this. In the case
when the geodesic flow is hyperbolic, like on a compact Riemannian manifold 
of strictly negative curvature, the density would follow. 

\begin{conj}
For any smooth Riemannian metric on $\mathbb{T}^d$ with $d>1$, the wave front of any 
point $P$ becomes dense. 
\end{conj} 

\begin{figure}[!htpb]
\scalebox{0.14}{\includegraphics{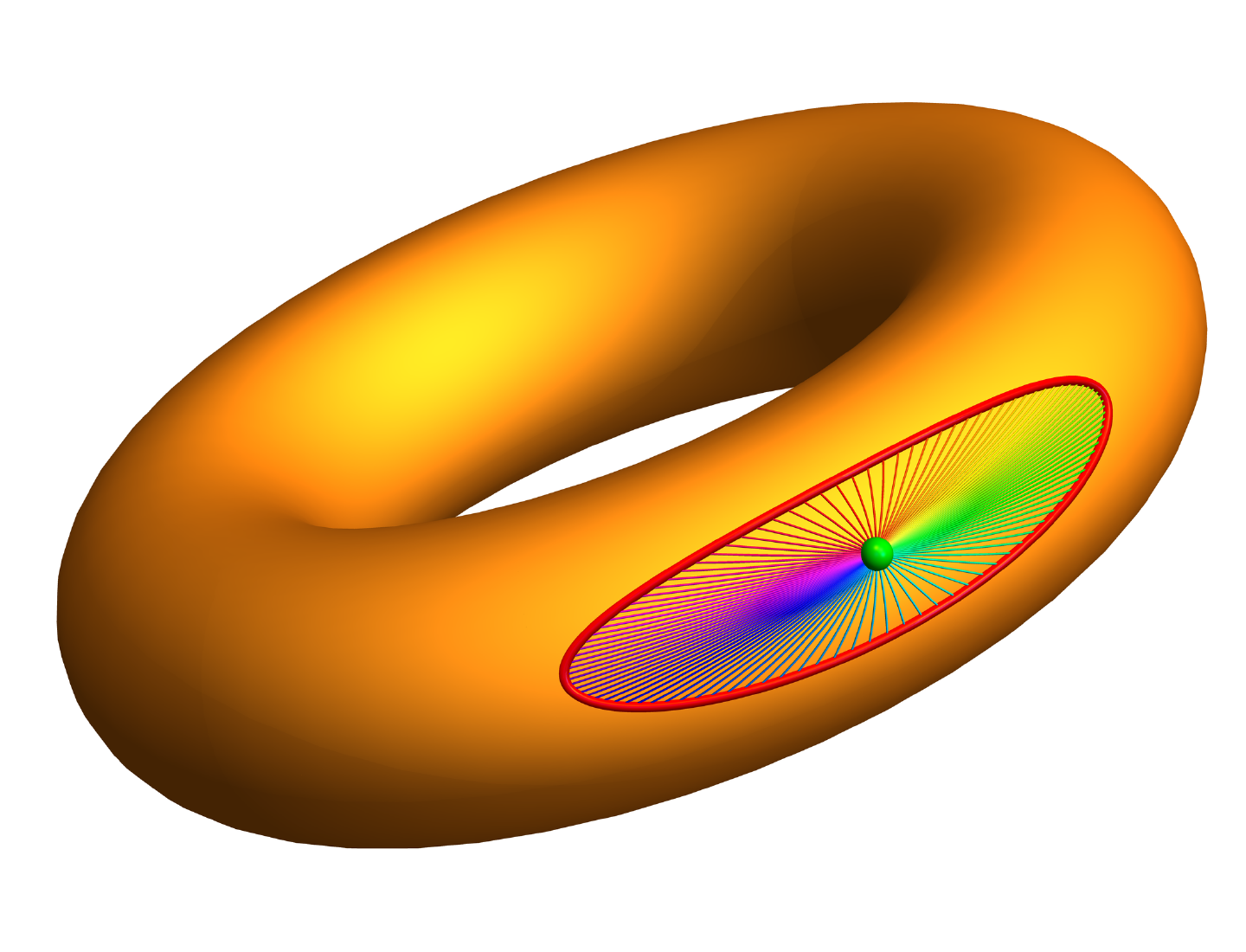}}
\scalebox{0.14}{\includegraphics{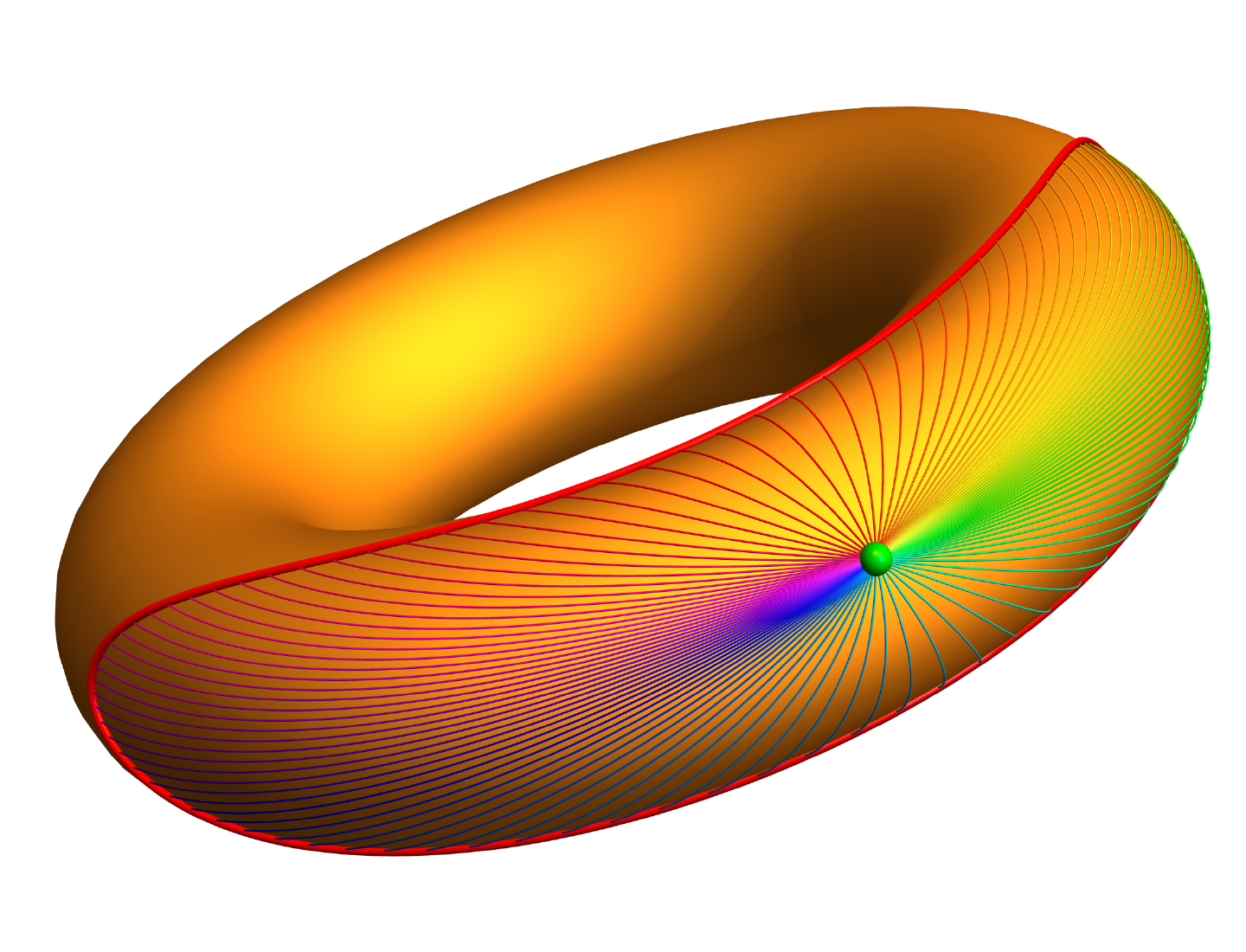}} \\
\scalebox{0.14}{\includegraphics{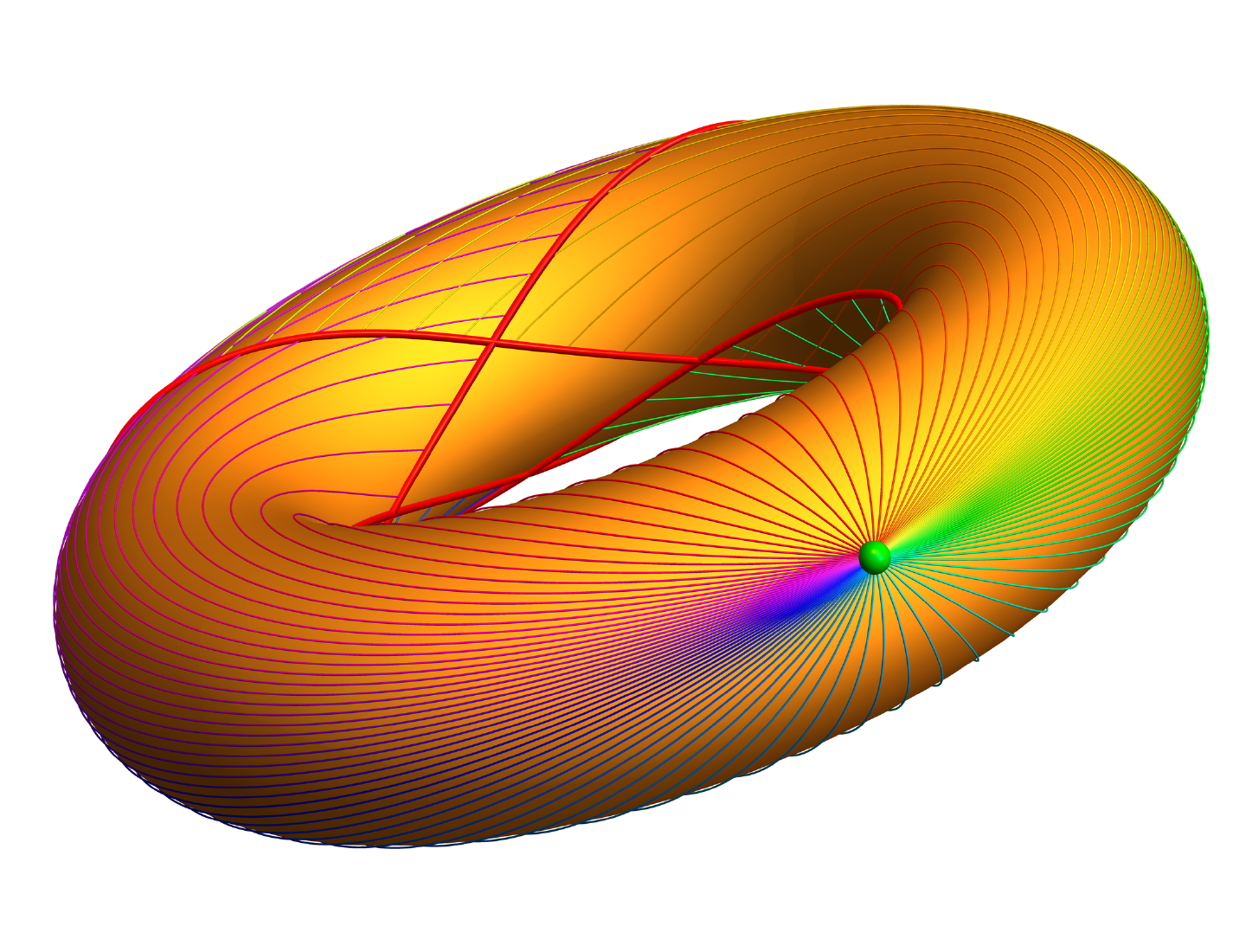}}
\scalebox{0.14}{\includegraphics{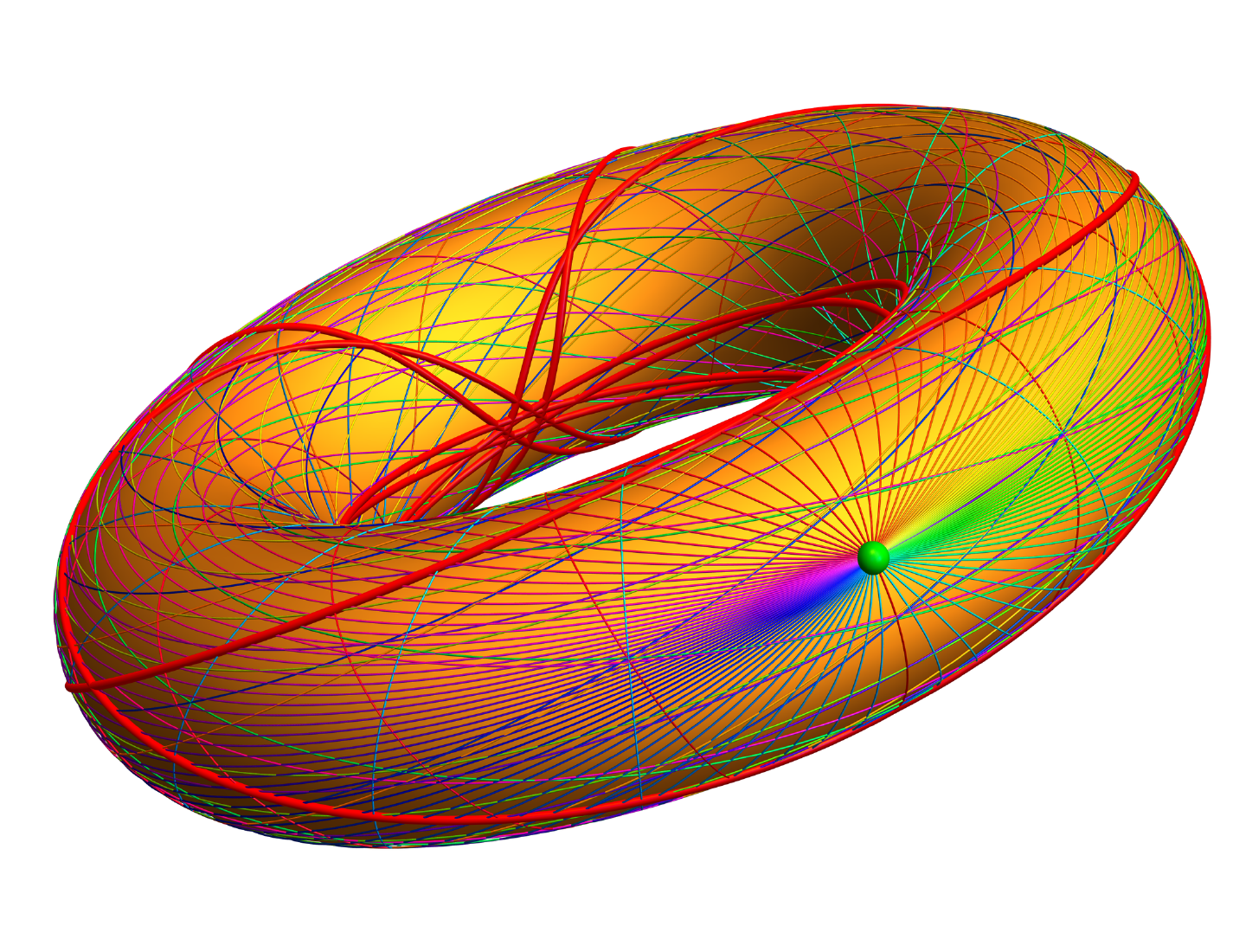}}
\label{Wave Fronts on Tori}
\caption{
Wave fronts $W_t$ on a flat torus. We prove that 
they become dense.
}
\end{figure}  

\begin{figure}[!htpb]
\scalebox{0.14}{\includegraphics{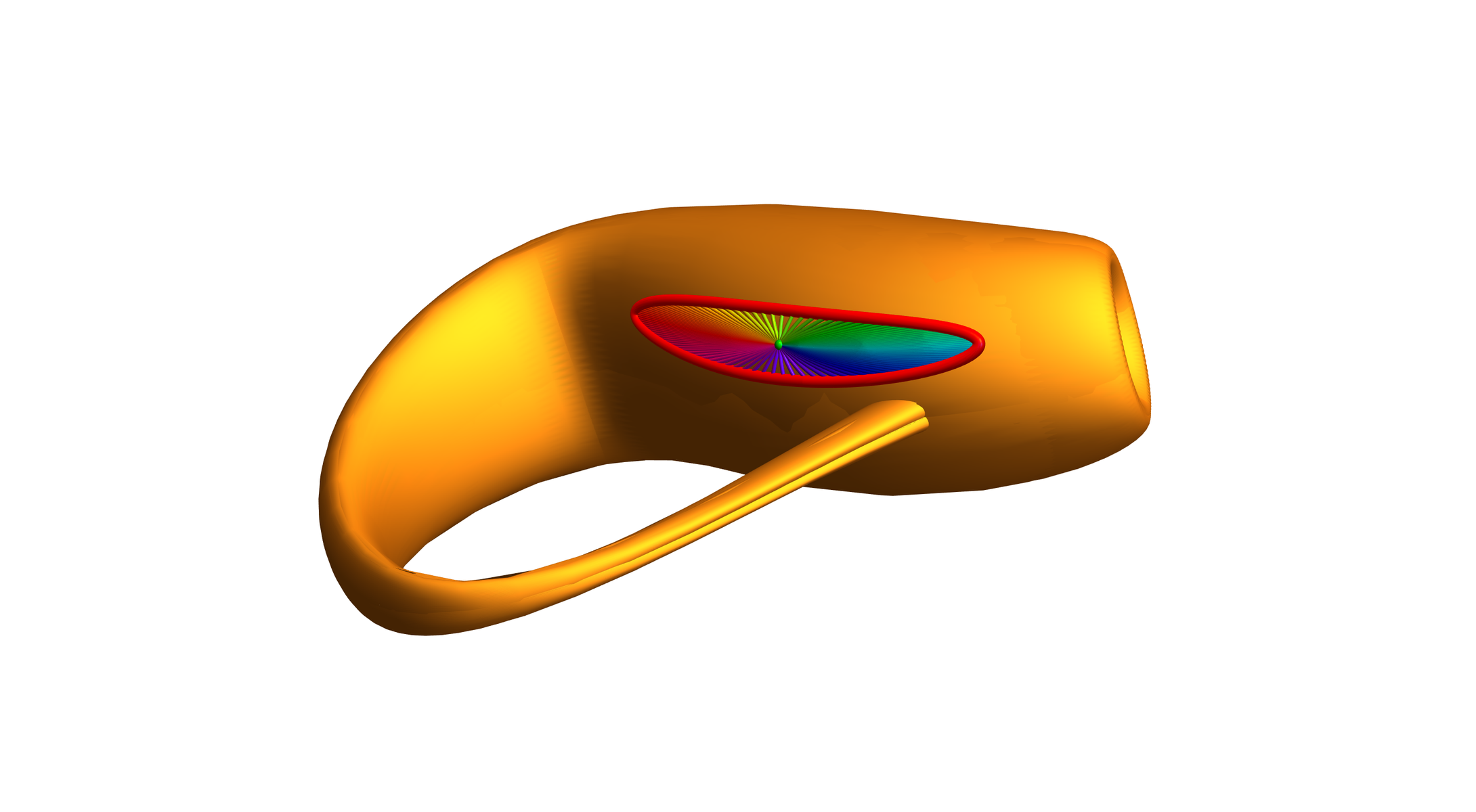}}
\scalebox{0.14}{\includegraphics{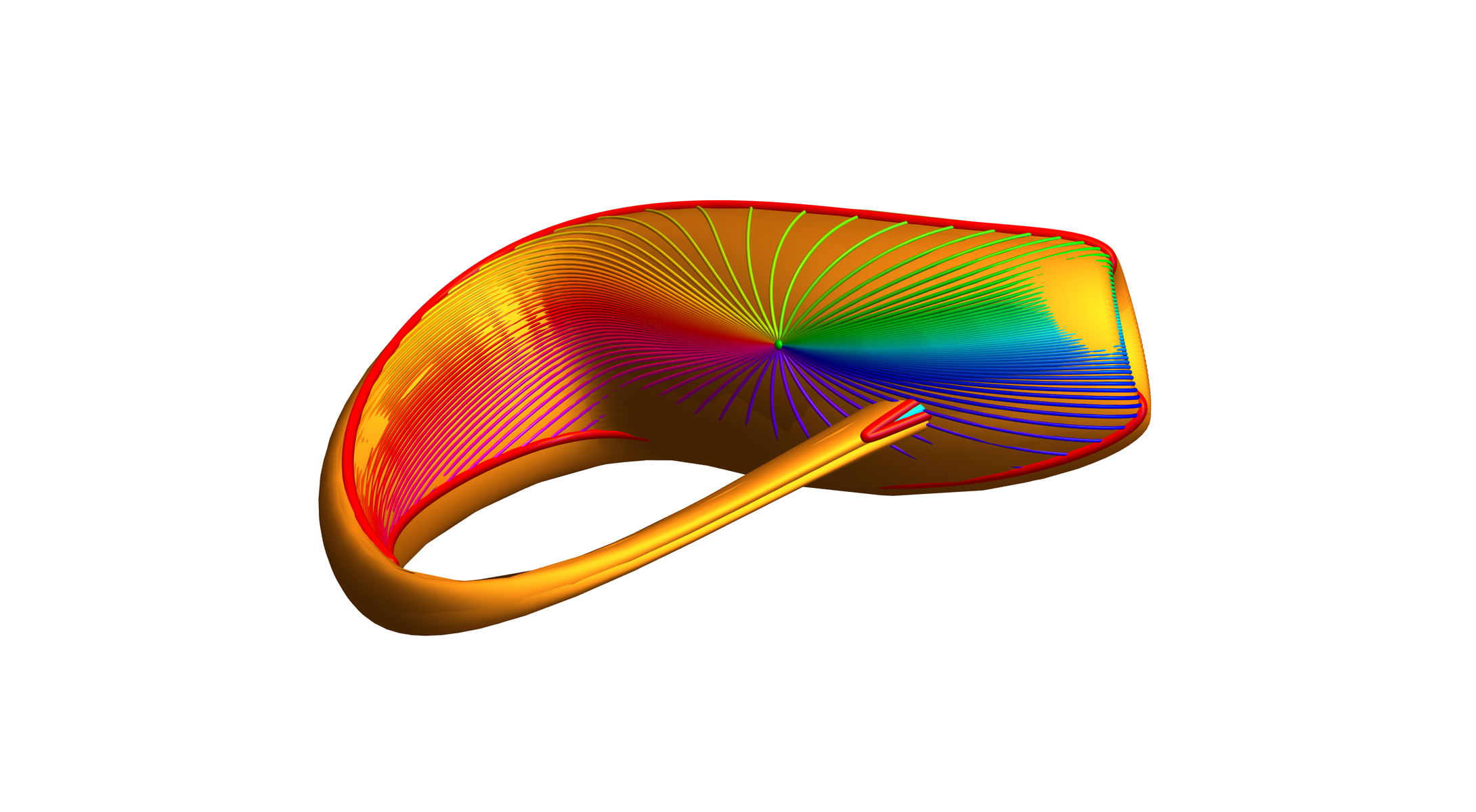}} 
\scalebox{0.14}{\includegraphics{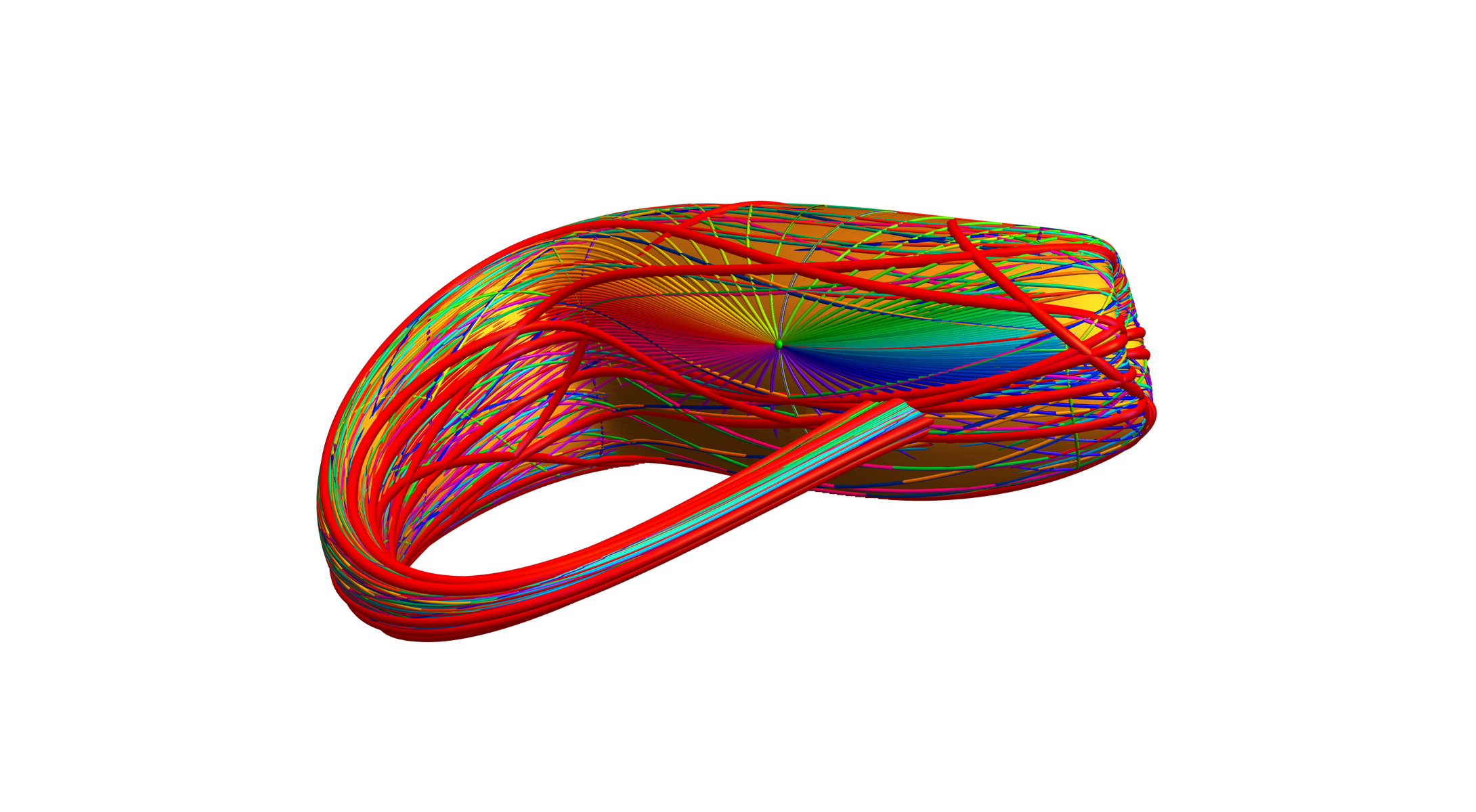}}
\scalebox{0.14}{\includegraphics{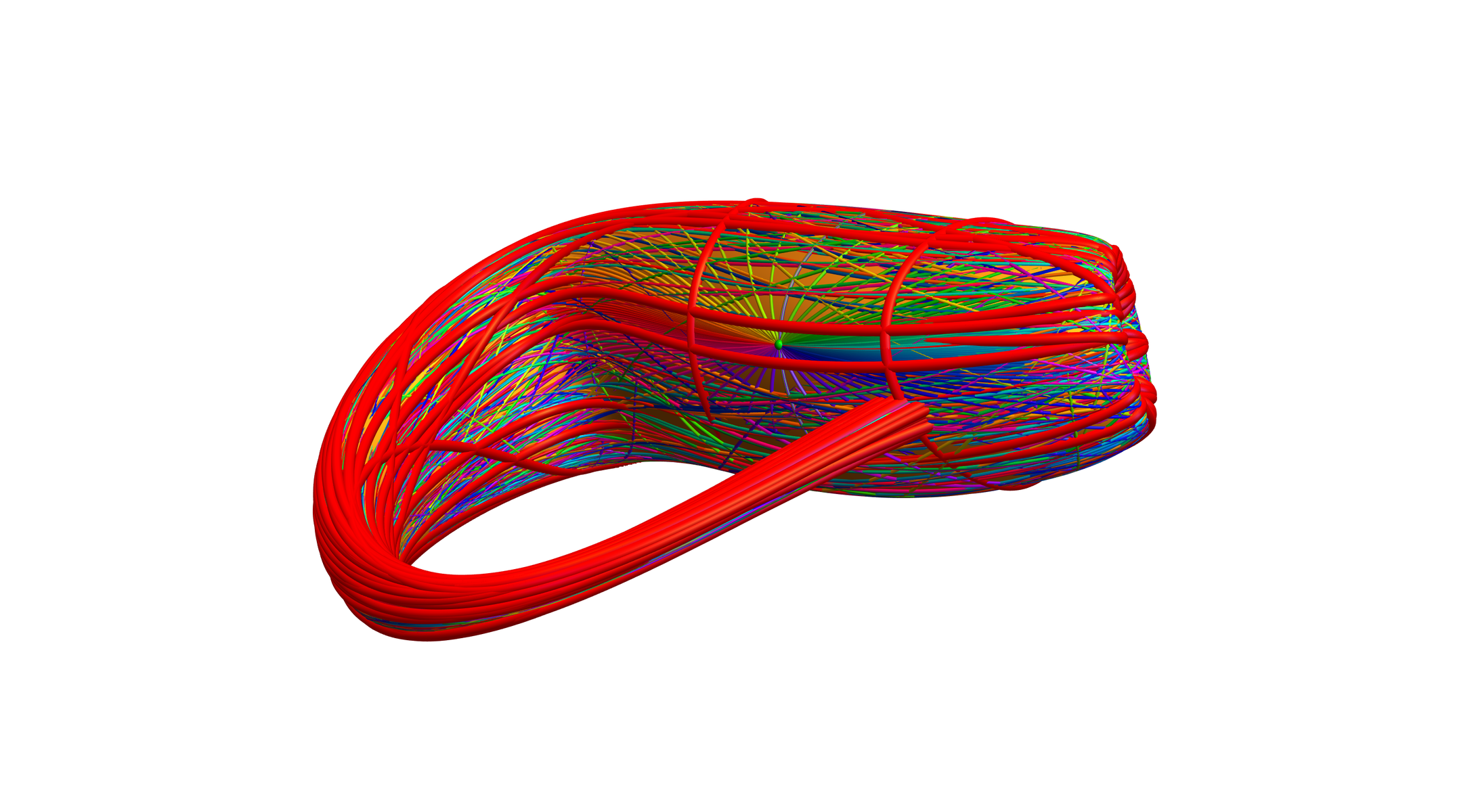}}
\label{Wave fronts on a Klein bottle } 
\caption{
Wave fronts $W_t$ on a flat Klein bottle. We prove that  
they become dense.
}
\end{figure}

\section{Billiards}

\paragraph{}
Billiards in convex tables were used by George Birkhoff
\cite{Birkhoff} (in section 6 of his book) and his postdoc Hillel Poritsky \cite{Por50} 
to illustrate principles and theorems in dynamical systems theory, especially in the context
of the 3-body problem. Periodic points are critical points of a variational principle.
The Birkhoff fixed point theorem even can establish the existence of many
periodic points if the billiard map. For modern texts on billiards, see 
\cite{KozlovTreshchev,Tab95}. 
A plethora of open problems persist: one does not know for example whether every triangular 
billiard has a periodic orbit. No smooth convex billiard exhibiting chaos 
(positive Lyapunov exponents on a set of positive measure in the phase space) is known. 
Candidates are tables like $x^4+y^4 \leq 1$. 
Even one of the first problem of Poritski \cite{Por50}, asking whether every integrable smooth 
convex billiard is an ellipse, remains unsettled in its full form despite
partial progress \cite{bialy2022birkhoff}.

\paragraph{}
A pair of points $p,q$ in $M$ is called a {\bf Wiedersehen pair}, if every geodesic through $p$
passes through $q$ and every geodesic through $q$ passes through $p$. If every point belongs to a Wiedersehen pair, then 
$M$ is called a ``Wiedersehen manifold" and must be a sphere as Blashke has conjectured.
Ellipsoids are examples of non-Wiedersehen manifolds, where some Wiedersehen pairs exist. In such situations, 
not all wave fronts become dense but most are believed to become dense. 
There are Wiedersehen situations also for billiards.
If we start a wave front at the center $P$ of a circular billiard table, then the 
wave front will periodically return to the point $P$. But already for the same 
circle, we expect the wave front $W_t(P)$ to become dense, if the starting point $P$ is not at 
the center. We are not aware even whether the density of the wave front in the 
simple duck problem of the circle has been shown. We can not use the method used here
for the torus.  In general, we expect:

\begin{conj}
For a generic convex billiard table, all billiard wave fronts $W_t(P)$ become dense. 
\end{conj} 

\paragraph{}
Also convincing is the following statement: 

\begin{conj}
For all convex billiard tables, there is a point $P$ such that the wave front $W_t(P)$ becomes dense. 
\end{conj}

\paragraph{}
Here is a problem that we feel could be settled with a simple argument, but which also 
escapes us so far:

\begin{conj}
For all convex polygon billiards, all the wave fronts $W_t(P)$ become dense.
\end{conj}

\paragraph{}
While the {\bf regular equilateral triangle}, the {\bf regular hexagon} the {\bf regular octagon}
or a right angle triangle, can be shown to satisfy the density condition, we do not know what happens already for 
a {\bf regular pentagon}. It might even hold for non-convex case as long as the illumination problem is solved there, meaning
that union of wave fronts covers the entire table up to a finite set of points. The Penrose table is a non-convex piecewise
smooth region with elliptic, circular and straight line boundary parts. No wave front now becomes dense there because placing
a light source at any point $P$ leads there to dark regions that can not be reached. 

\begin{figure}[!htpb]
\scalebox{0.68}{\includegraphics{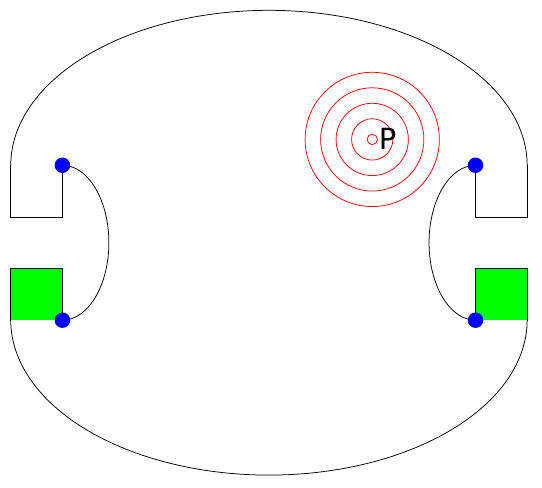}}
\label{counterexamples}
\caption{
In the Penrose table, there are for every $P$
(at least) two square parts of the table which are not reached by $W_t(P)$.
}
\end{figure}

\paragraph{}
As a consequence of the torus case, we can state the following result:

\begin{coro}
For the billiard in a square, the wave front of any point $P$ becomes dense. 
\end{coro}

\begin{proof}
By unwrapping the table we get to a universal cover $\mathbb{R}^2$ of $Q=[0,1] \times [0,1]$.
The circle $S_t(P)$ centered at $P$ gets mapped into $K_t(P) = \pi(S_t(P))$, where $\pi$
is the map which first maps $(x,y)$ to $(x,y) \; {\rm mod} \; 2$, then applies the map 
$(x,y) = (q(x),q(y))$, where $q$ is the tent map $q(x)=x$, if $x \leq 1$ and $q(x)=2-x$ 
if $x \in [1,2]$. In short, the piecewise smooth exponential map 
$$  \pi(x,y) = (1-|(x \; {\rm mod} \; 2)-1|,1-|(x \; {\rm mod} \; 2)-1|) $$
maps the circle $S_t(P)$ to $K_t(P)$. 
\end{proof}

\begin{figure}[!htpb]
\scalebox{1.2}{\includegraphics{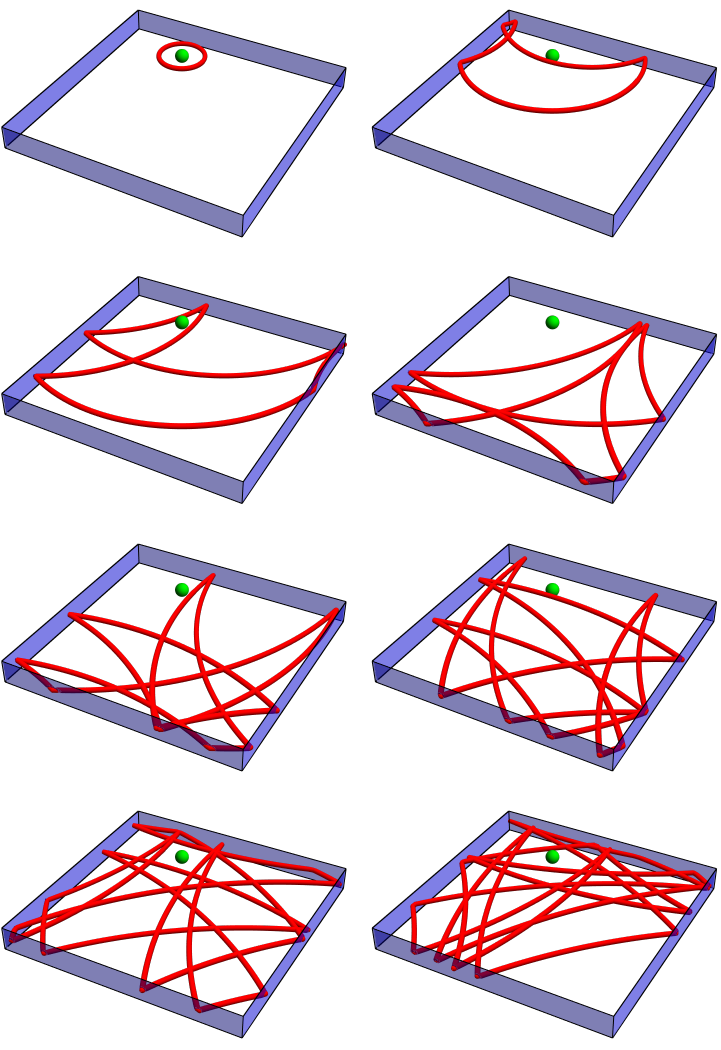}}
\label{Square billiard}
\caption{
Wave fronts $W_t$ in a square billiard. We prove that $W_t$ becomes
dense on the table. 
}
\end{figure} 

\paragraph{}
We were unable to find references which even {\bf ask} such wave density 
questions. But we suspect that it must have been on the mind of many 
researchers in the field, as some numerical simulations of wave front beams have been done
in \cite{berglund2021square,berglund2021uncentered,
berglund2021pentagon}. These animations show the beauty and complexity 
which can happen with wave fronts, even for simple situations like a 
circle or an ellipse. Related to wave fronts are caustics and the 
``Jacobi last theorem" type question of how many cusps such a caustic has.
See \cite{bor2024cusps}. 

\paragraph{}
Nils Berglund kindly informed us about 
\cite{Vicente2020,VerdiereVicente2020,VerdiereVicente2021},
where the question of the {\bf length of the wave fronts} is addressed. 
Vicente visualized the problem as ``Une goutte d'eau dans un bol", ("a drop of
water in a cup"). \cite{Vicente2020} showed that for the duck problem in the
unit disk, the length $|W_t|$ satisfies asymptotically 
$|W_P(t)| = \arcsin(|P|) t + o(t)$, in agreement with
$|W_O(t)|=O(1)$ for the origin $O$ and that the longest wave fronts are produced
at boundary of the pond. The length question is not so interesting on the flat torus
or rectangle or cube, where $|W_P(t)|=2\pi t$ for all initial points. One
would naively expect at first $|W_P(t)|-|W_Q(t)|=o(t)$ to hold in general for a pair of
points. The pond case shows that this is not true. It begs the question whether there
is an other 2-manifold with boundary where the growth rate $|W_P(t)|/t$ depends on $P$. 

\paragraph{}
One possible attack to the density of wave front problem is to focus
on {\bf periodic points} which exist for example in a Birkhoff billiard
and which are very much also under investigation in regular polygons
\cite{DavisLelievre}.

\section{Polyhedra}

\paragraph{}
The starting point of this investigation was a numerical simulation of the 
wave front on the unit cube. The second author has made movies of wave fronts
on the cube since 2011 for multi-variable courses to illustrate how quickly
simple open questions can appear in mathematics. A wave front on the 
torus can be seen on \cite{CFZ}. The problem of finding the smallest
connection between two points of a cube or to find the path of a billiard shot 
in a rectangle reaching from a point to one of the corners by reflecting a given
number of times on walls can all be solved by unwrapping the situation. This 
produces nice puzzles in the context of distance questions in multi-variable 
calculus courses, like finding the shortest connection between two points on a cube
or the length of a shortest billiard path in a rectangle that is required to 
bounce a pre-required number of times on each wall. 

\paragraph{}
We were fascinated to see that in the cube case, the 
wave fronts became more and more complicated and that the wave fronts $W_t(P)$ even 
lost connectivity. Passing through a corner can cut the wave front because the 
geodesic flow was no longer continuous there. In the universal cover of the cube, 
once several lattice points were crossed, the wave front already had a chance
of been hacked into pieces. Computer experiments indicated that the wave front becomes dense
even if we start at a corner. The biggest surprise was maybe that the lack of smoothness of the exponential map 
could leads to situations on polyhedra, where wave fronts become disconnected and complicated 
already after a relatively short time. 

\begin{coro}
On the surface $M$ of a cube, all wave fronts become dense.
\end{coro}

\paragraph{}
Fix $P \in M$. It is allowed to be one of the corners of the cube $M$. 
Let $Q$ be a second point on the cube and assume $\epsilon>0$ is given. 
We will show that for every
$t>144/\epsilon^2$, the wave front $W_t(P)$ intersects with the $\epsilon$ 
neighborhood of $Q$. Each $P$ defines a set of measure zero on $M$ such that 
if $Q$ is not there, then any of the countably many possible geodesic connecting
$P$ with $Q$ is defined for all times and does not pass through any of the corners of $M$.

\paragraph{}
The {\bf net} $N$ of the cube $M$ is defined as a union of 6 translated squares 
$[k,k+1) \times [l,l+1)$ in $\mathbb{R}^2$. We can draw it as a union of squares $A,B,C,D,E,F$ in 
what we call the {\bf large torus} $[0,4) \times [0,4)$.
Without loss of generality, the origin of the wave $P$ can be chosen to be in $A$. As there is a 
folding correspondence between $N$ and $M$, there is a point $Q$ on $N$ that belongs to $Q$ in $M$. 

\paragraph{}
Theorem~1 showed that for $3/\sqrt{t}< \epsilon/2$, the wave front
$W_t(P)$ on the small torus $\mathbb{T}^2/\mathbb{Z}^2$ 
covers $\mathbb{T}^2$ $(\epsilon/2)$-dense.
Scaling implies that for $6/\sqrt{t/4}<\epsilon$, every $(\epsilon/2)$ -
neighborhood of $W_t(P)$ covers the torus $\mathbb{R}^2/(4 \mathbb{Z})^2$ of side length $4$. 
This means that for every $t>36/(\epsilon/2)^2$, the set $B_{\epsilon/2}(Q) \cap W_t(P)$ is not empty.
The wave front $W_t(P)$ on the large torus is now $\epsilon/2$ dense. It is therefore
$(\epsilon/2)$-dense also on each of the squares $\{A,B,C,D,E,F\}$.

\paragraph{}
Draw the {\bf circular wave front} $S_t(P) = \{X, ||X-P|| =t \}$ on the cover $\mathbb{R}^2$ of the large torus. 
To every $X \in S_t(P) \subset \mathbb{R}^2$ for which $PX$ 
does not pass through a lattice point, we can associate a group element $\gamma(X)$ in the 
{\bf octahedral group} $G$ that is isomorphic to the group of permutations of the 4 space diagonals.
$\gamma(X)$ tells to which of the 24 oriented faces the point $X$ belongs in $M$ and so in $N$.
The map $\gamma: S_t(P) \to G$ is complicated. Figure~(10) illustrates this 
Almost every $X \in S_t(P)$, when lifted on $S_t(P) \times G$, encodes
where the corresponding $W_t(P)$ is on the cube and where it is located in the net $N$. 
Crossing the boundary of a face updates the group. Crossing from $A$ to $B$ defines a group element $g_{A,B}$.
This happens for every pair of squares $A',B' \subset \mathbb{R}^2$ which satisfy $\pi(A')=A,\pi(B')=B$. 
Rotations could be realized by transitioning around a corner $B \to A \to F \to B$. Again, this
also happens on any of the lifts in $\mathbb{R}^2$: for example if $X=(x,y)$ is given somewhere on $S_t(P)$ and $\pi(X) \in A$
with $\gamma((x,y))=g$, then $\gamma((x+1,y)) = g_{A,B} \circ g$. 
(The program that computed Figure~5 did evolve the real line $t(\cos(\alpha),\sin(\alpha))$ for a 
hundred thousand directions $\alpha$ and kept track of the group element in the form of a 
face number and basis to draw the wave onto the faces.) 

\paragraph{}
Almost every point $X \in W_t(P) = \pi(S_t(P))$ on the wave front in the large torus $[0,4) \times [0,4)$ is now associated 
to a unique group element $g=\gamma(X) \in G$. It especially is defined for every point in $N \cap W_t(P)$ on the net $N$. 
The pair $(X,g)$ allows to recover the position of the corresponding point on the cube $M$. 
We know that for $Q \in N$, there is a point $X$ on the toral wave front $W_t(P)$ in $N$, such 
that $||Q-X|| \leq \epsilon/2$. This point $X$ might correspond to the wrong oriented face on the cube.
But $X$ being on the wave front, it has a group element $\gamma(X) \in G$ 
attached. For every $g \in G$, there is an other point $X_g$ in $N$, which is a translated and rotated version of $X=X_{\gamma(X)}$. 
While the points $\{X_g\}_{g \in G}$ are not on $W_t(P)$ in general, there is for each $X_g$ a point $Y_h \in N \cap W_t(P)$,
such that $||X_g-Y_h|| \leq \epsilon/2$. Now $||Q-Y_h||  \leq ||Q-X_g|| + ||X_g-Y_h|| \leq \epsilon/2 + \epsilon/2 = \epsilon$ in $N$ and 
$Y_h$ is on the wave front $W_t(P)$ in $N$. If $N$ is folded onto $M$ isometrically, we see the point $Y_h$ on $M$ that is 
$\epsilon$ close to $Q$. 

\begin{figure}[!htpb]
\scalebox{0.2}{\includegraphics{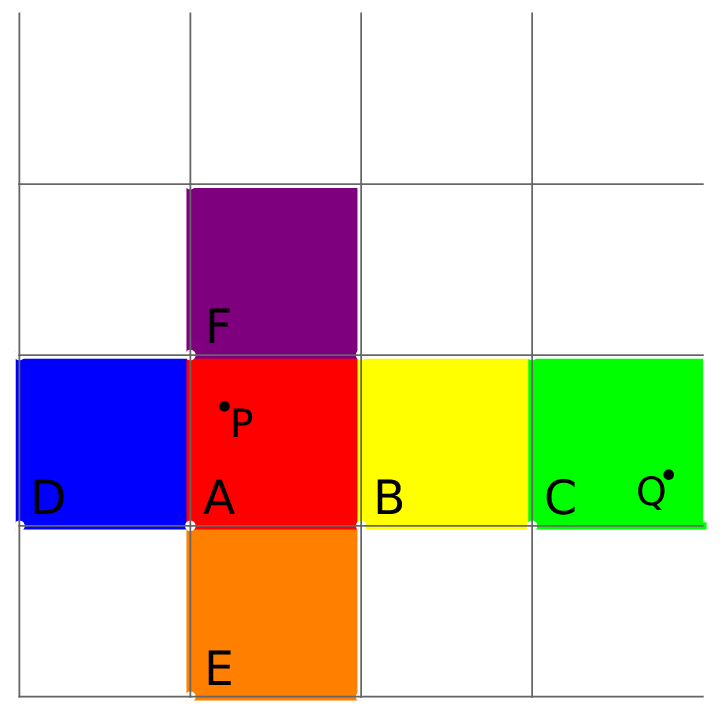}}
\scalebox{0.08}{\includegraphics{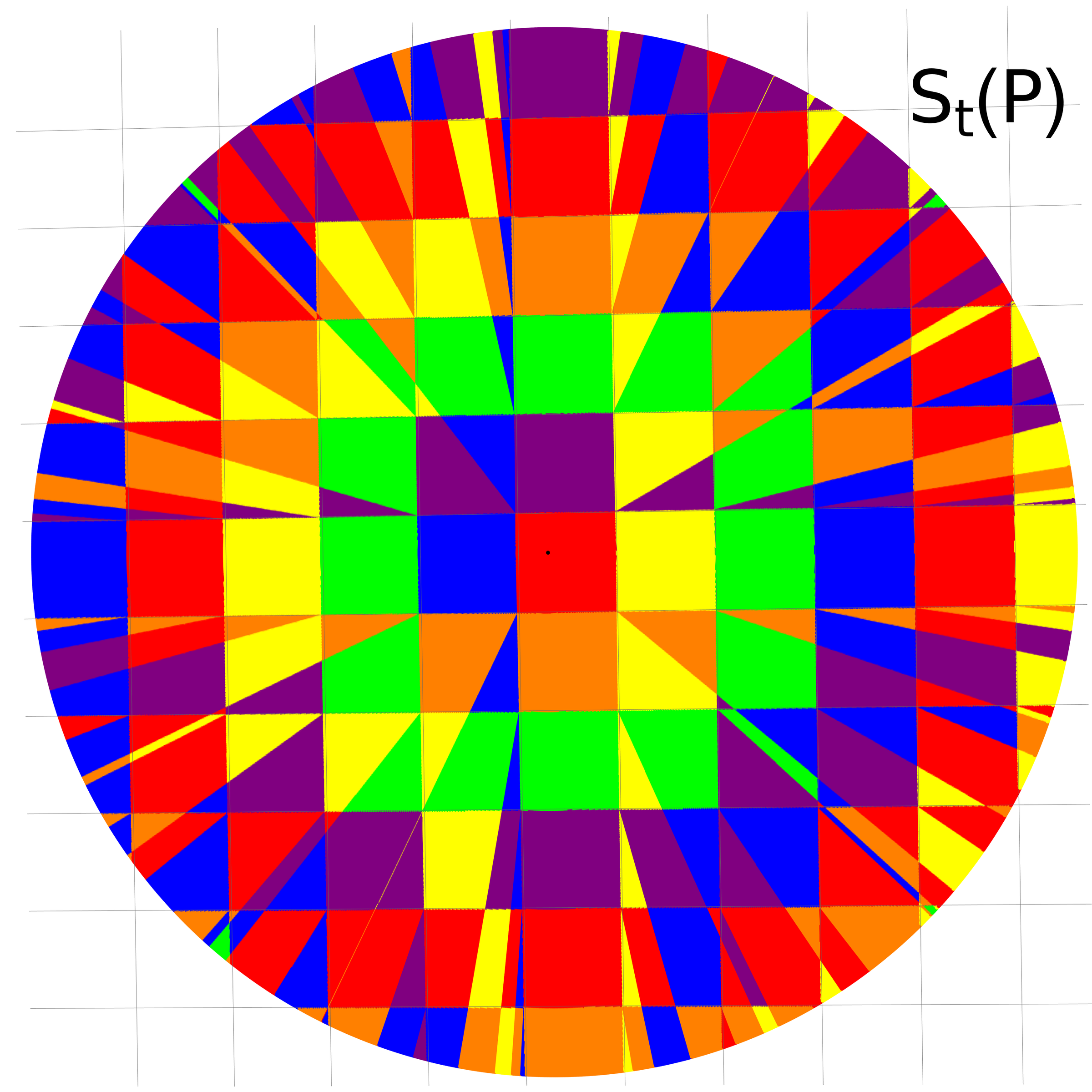}}
\label{cubeexplanation}
\caption{
To the left, we see the net $N$ inside the large torus $[0,4) \times [0,4)$.
The net $N$ is $M$ that is cut and flattened onto $\mathbb{R}^2$.
To the right, the color on an eligible line through $P$ encoding the face
number of $M$.
}
\end{figure}

\begin{figure}[!htpb]
\scalebox{0.33}{\includegraphics{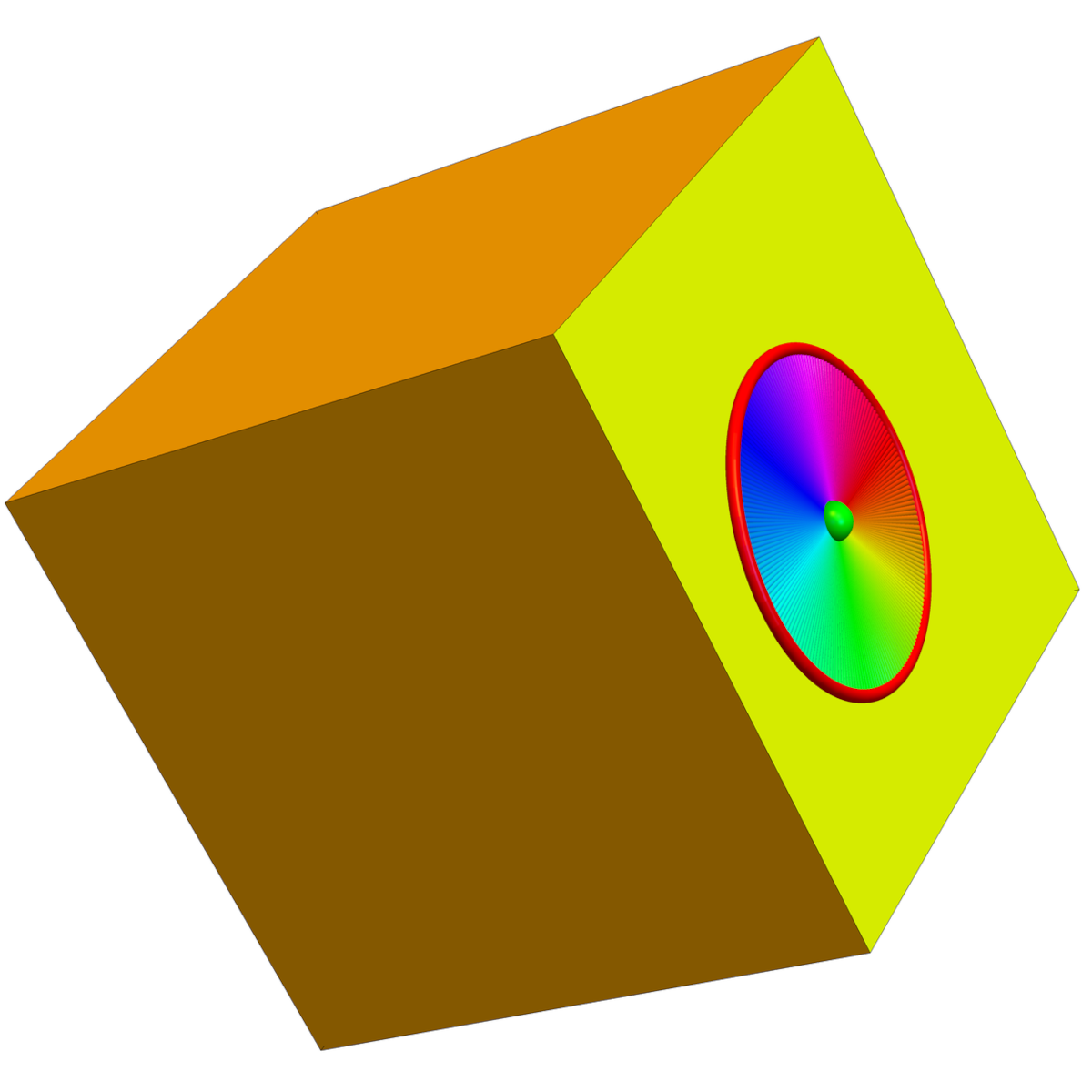}}
\scalebox{0.33}{\includegraphics{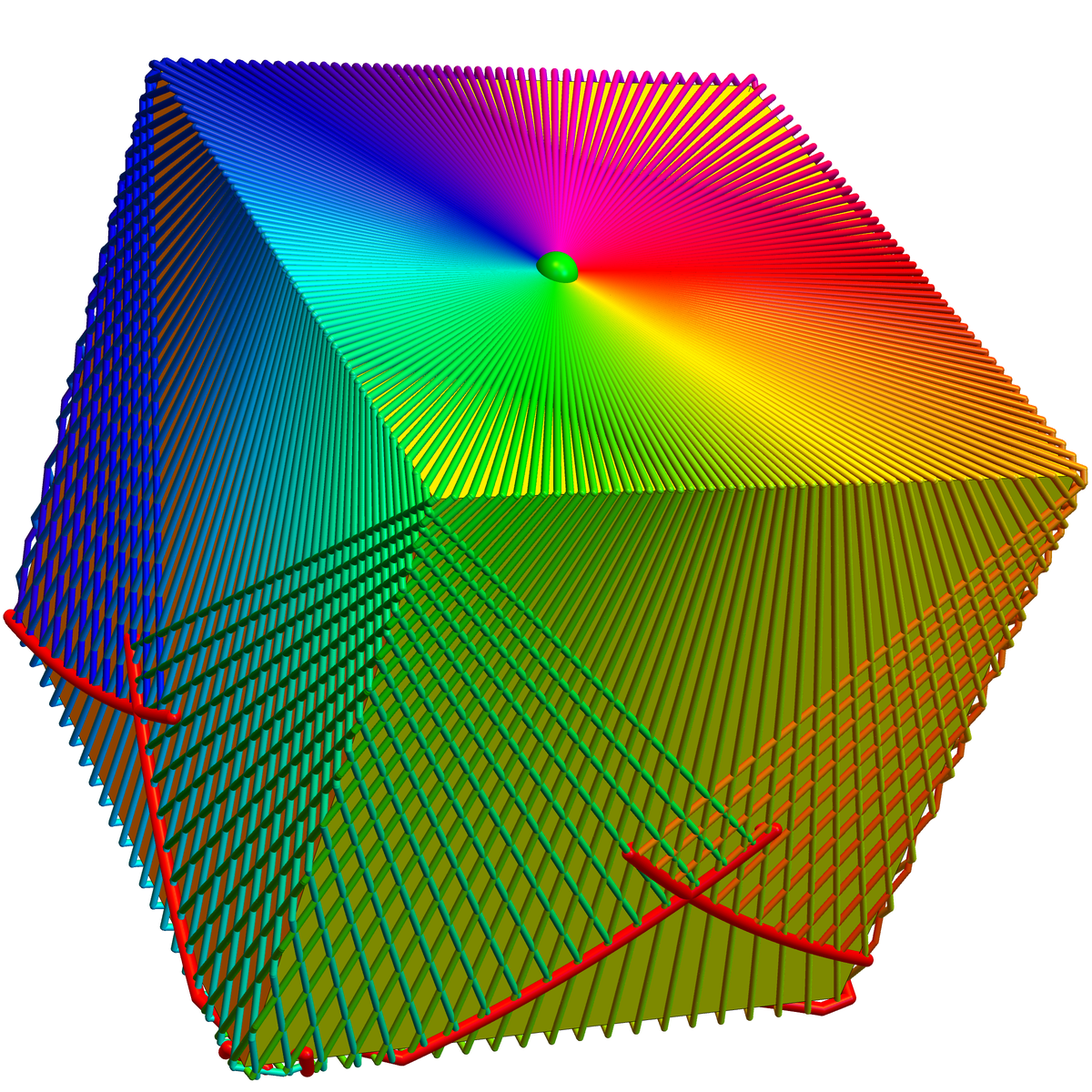}}
\scalebox{0.33}{\includegraphics{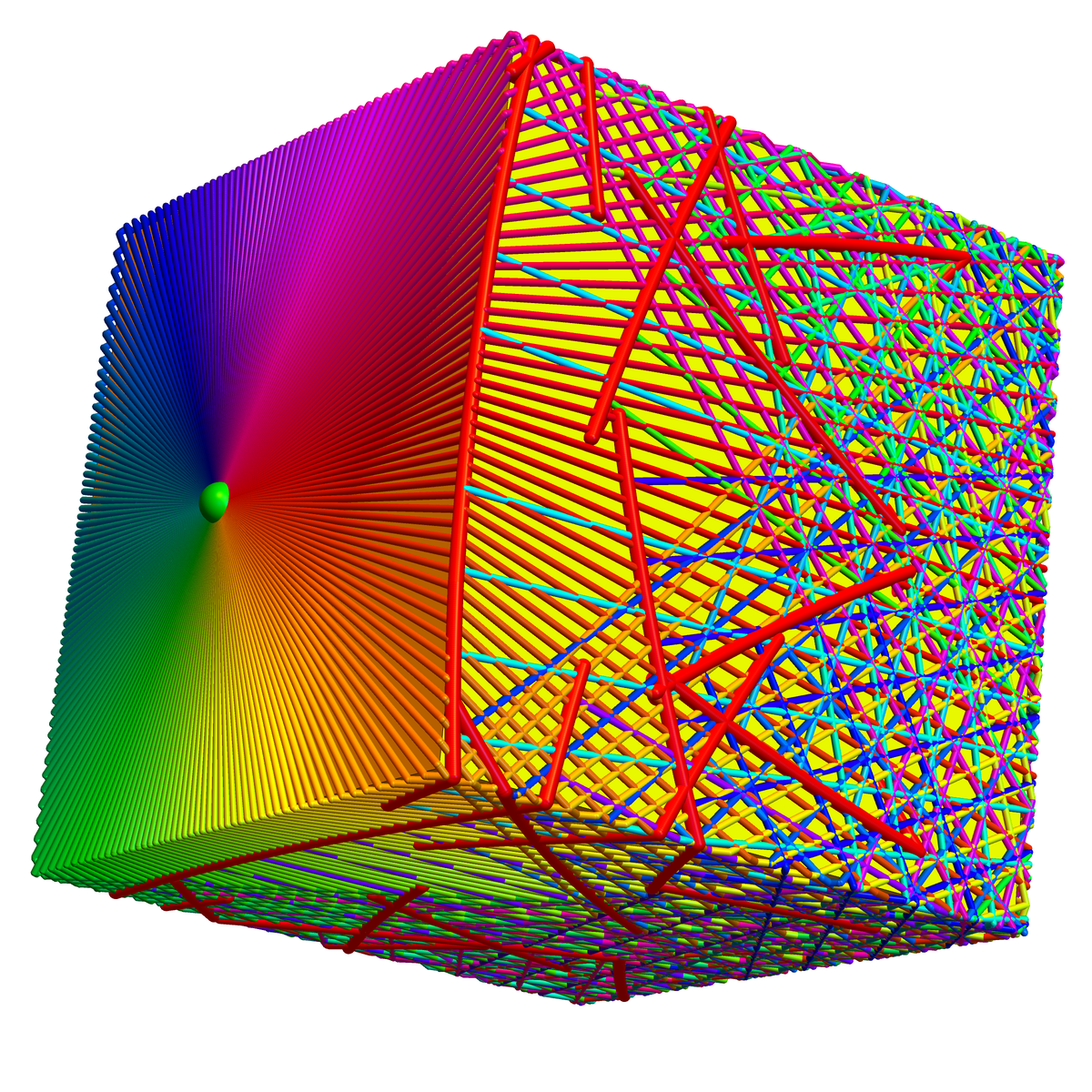}}
\scalebox{0.33}{\includegraphics{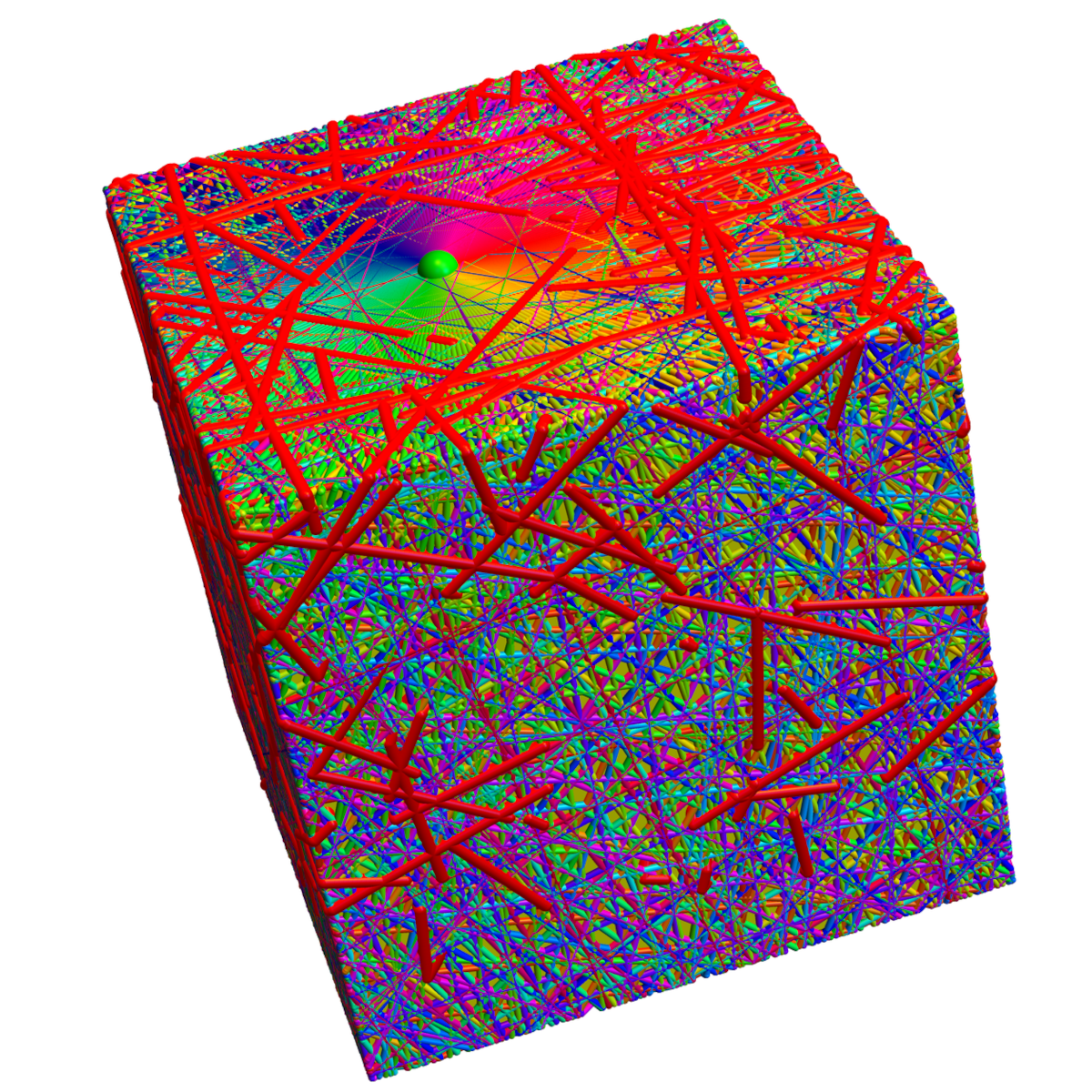}}
\label{Cube Wave Fronts}
\caption{
Wave fronts $W_t$ on a cube. We prove that 
they become dense. 
}
\end{figure}  

\section{Existing Literature}

\paragraph{}
The study of {\bf wave front propagation} in Riemannian geometry
intersects with other areas of mathematics like geometric number 
theory, partial differential equations and dynamical system theory.
The case of a torus also connects with the geometry of numbers
or Ehrhart theory. These theories that shed light on the distribution and density of 
lattice points if the solid is scaled. 
As an example, we mention Huxley and Nowak's work on primitive lattice points in convex planar 
domains \cite{huxley1996primitive} that provides a detailed analysis of the distribution 
of these points within a compact convex subset of $\mathbb{R}^2$. 
There are intricate relationships between the geometry of the domain and the associated 
error terms. This study offers foundational insights that can be extended to understand 
how wave fronts might distribute over more complex surfaces.
In the geometry of numbers \cite{lax2000geometry}, one explores the relationship between 
geometric structures and lattice point distributions within the context of dynamical systems, 
particularly focusing on how curvature and boundary shape influence the number of lattice points. 
Waxman and Yesha's recently looked at lattice points within thin sectors \cite{waxman2023number}. 
How does the number of integer lattice points within shrinking sectors of a circle grow? Such
problems are situated at the intersection of analytic number theory and geometric analysis. The 
Diophantine properties of the slope of the sector's central line affect the distribution of lattice points. 
Such studies could help to understand wave front behavior in more constrained geometric settings
such as rectangular billiards or on the surface of a cube, where the geometric structure imposes 
additional complexity.

\paragraph{}
The problem of wave fronts is also related to the {\bf illumination problem}: what
is the set of points $Q$ such that the wave front $W_t(P)$ never reaches $Q$? 
This question was first asked the 1950s. In 1958, Roger Penrose constructed the 
{\bf Penrose un-illuminable room}. In 1995, the first polygonal unilluminable
room was found \cite{tokarsky1995polygonal}.
For a general polygon, this is not an empty set in general, but it is maximally a finite set of points 
in general \cite{wolecki2024illumination}.
For recent work on translation surfaces see \cite{LelievreMonteilWeiss2016}. 
(Added: January 26th: The Penrose example already shows that there are regions for which entire
parts of the room cannot be illuminated, implying the wave front can not
be dense. Thanks to Sergei Tabachnikov to point this out to us.) 
On the other hand, having a room that can be illuminated does not
prove that fronts become dense. The duck in the center of a 
circular pond illustrates this. 

\paragraph{}
An obvious relation of wave front dynamics is with partial differential equations,
like wave dynamics. If we have a Riemannian manifold we can look at the wave 
dynamics on it. If a solution to this partial differential equation 
starts localized at a point, then the wave front is related to the 
wave front given by the geodesics. This is related to the {\bf Huygens principle}. Wave dynamics
is also closely linked to quantum mechanics because if $L$ is the Laplacian on the 
manifold and $L=D^2$, then the wave equation $u_{tt} = -Lu$ is by d'Alembert
equivalent to two Schr\"odinger type equations $u_t = -i Du$ and $u_t=+iD u$ for the 
square root $D$ of the Laplacian. This motivates semi-classical analysis 
\cite{cardinzanelli2017schrodinger}. There are connections between high energy 
eigenfunctions of the Laplacian and the geodesic or billiard dynamics
\cite{jakobson1997laplacian}. We are not aware of any relation between the wave front
spread and the eigenfunctions of the Laplacian. 

\paragraph{}
In Riemannian geometry, a manifold is called a {\bf Wiedersehen manifold} if for every point $P$
there is a {\bf Wiedersehen point} $Q$ such that all geodesics through $P$ pass also through $Q$.
Wilhelm Blaschke conjectured in 1921 \cite{blaschke}, that every
Wiedersehen manifold is a sphere. Only in 1963, Green proved that a 2-dimensional Wiedersehen manifold
(Wiedersehenfl\"ache) must be isometric to a sphere of constant curvature. 
See \cite{Klingenberg1978} Theorem 5.2.5. 
Note that for a projective plane of constant curvature, 
the point $Q$ is the same than $P$ so that it is technically not called a
Wiedersehen manifold. A natural question is whether it is true that for any non-Wiedersehen manifold (different
from a projective space situation), there exists a point $P$ for which the wave front $W_t(P)$ 
becomes dense and that for any d-manifold with or without boundary that does not allow for a 
$SO(d-1)$ symmetry, the wave front becomes dense for all points $P$.
An example of a $d=2$-manifold with boundary is the unit ball which carries a $SO(1)$ rotational
symmetry. An other example is an ellipsoid of revolution which again carries a $SO(1)$ rotational
symmetry and for which there exist points $P$ which have a Wiedersehen point $Q$. 

\paragraph{}
We have already mentioned that in general, we expect wave fronts
to become dense. What about specific situations in higher dimensions, like in the case
of billiards in a regular polytop? In the case of a polytop that 
tesselates Euclidean space, we can prove the density by projection.
How can one prove for a regular polytop that $W_t(P)$ becomes dense?  
A good place to start would be a solid regular dodecahedron.
The proofs given here do not apply.
Still, we feel that there should be a simple argument proving that 
wave fronts are dense in {\bf any convex polytop} of dimension 2 or larger. 

\paragraph{}
Already in the flat case, billiards can show chaotic motion. 
The wave front length then can grow exponentially in time. We expect
the wave front in such a case to be $e^{-C t}$ dense for some constant $C$.
Can one prove this? A case to consider would be the Bunimovich \cite{Bun79}
stadium, which is the boundary of the convex hull of the union 
of two circles of equal radius which do not have the same center. 
We would expect there that there is a $c>0$ such that 
every ball of radius $B_{e^{-c t}}$ intersects $W_t(P)$ for large
enough $t$. This does not follow from the strong mixing property of the 
billiard flow. Strong mixing means that for two sets $U,V$ of positive area in the phase
space (the area is normalized to $1$ in order to have a probability space) that
$\lim_{t \to \infty} |T_t(U) \cap V| = |U| |V|$. In probability theory, one
would rephrase mixing that the random variables $1_{T_t(U)}$ becomes uncorrelated to
the random variable $1_V$ in the limit $t \to \infty$. 

\paragraph{}
An other case, where we expect that a simple argument could 
show density of $W_t(P)$ with a rate $e^{-Ct}$ are 
compact hyperbolic manifolds on which the geodesic flow has
long been known to be ergodic \cite{EberhardHopf1932}. 
In the hyperbolic case, by a theorem of Hadamard, the wave fronts do not develop 
singularities. There are metrics on the 2-sphere for which the 
flow is ergodic. This certainly should also indicate that the 
wave fronts become dense, but we are not aware of an argument 
to see this. 

\paragraph{}
There would be other questions. 
In the cube case we have seen that the fronts $W_t$ becomes disconnected
pretty fast. The total length of the curve of course is $2\pi t$.
What is the typical length of a connected component, where we
only consider {\bf connected components} $\exp_t(I)$ with $I$ an interval
in unit sphere $S \subset TM_P=\mathbb{R}^2$, such that 
$\exp_s(I)$ is a connected curve for all $s \leq t$?
These floating pieces get disconnected even more
but also start to intersect with other such curves.
Is it possible that the set $W_t$ can become connected again 
for some large $t$?

\section{More experiments}

\paragraph{}
A referee observed that we have given little numerical evidence for the conjectures. 
In order to illustrate the situation, we add two pictures. The first is a test
case in billiards. It is related to the dodecahedron case. The second is an 
example of a two-dimensional Riemannian manifold, where we chose a non-flat metric.
They are examples of many experiments we have done over
the years, both with billiards as well as with geodesic flows.

\paragraph{}
Let us look at the situation of a regular pentagon, where
the unfolding produces a Veech surface has genus 2 \cite{Davis2016}. 
Periodic trajectories have slopes in the golden field 
$\mathbb{Q}(\sqrt{5})$. All orbits are either periodic or uniquely ergodic. 
We feel that this example could be a case, where the wave front question is 
accessible. As in any polygon, wave front length is $|W_t(P)|=2\pi t$.

\begin{figure}[!htpb]
\scalebox{0.1}{\includegraphics{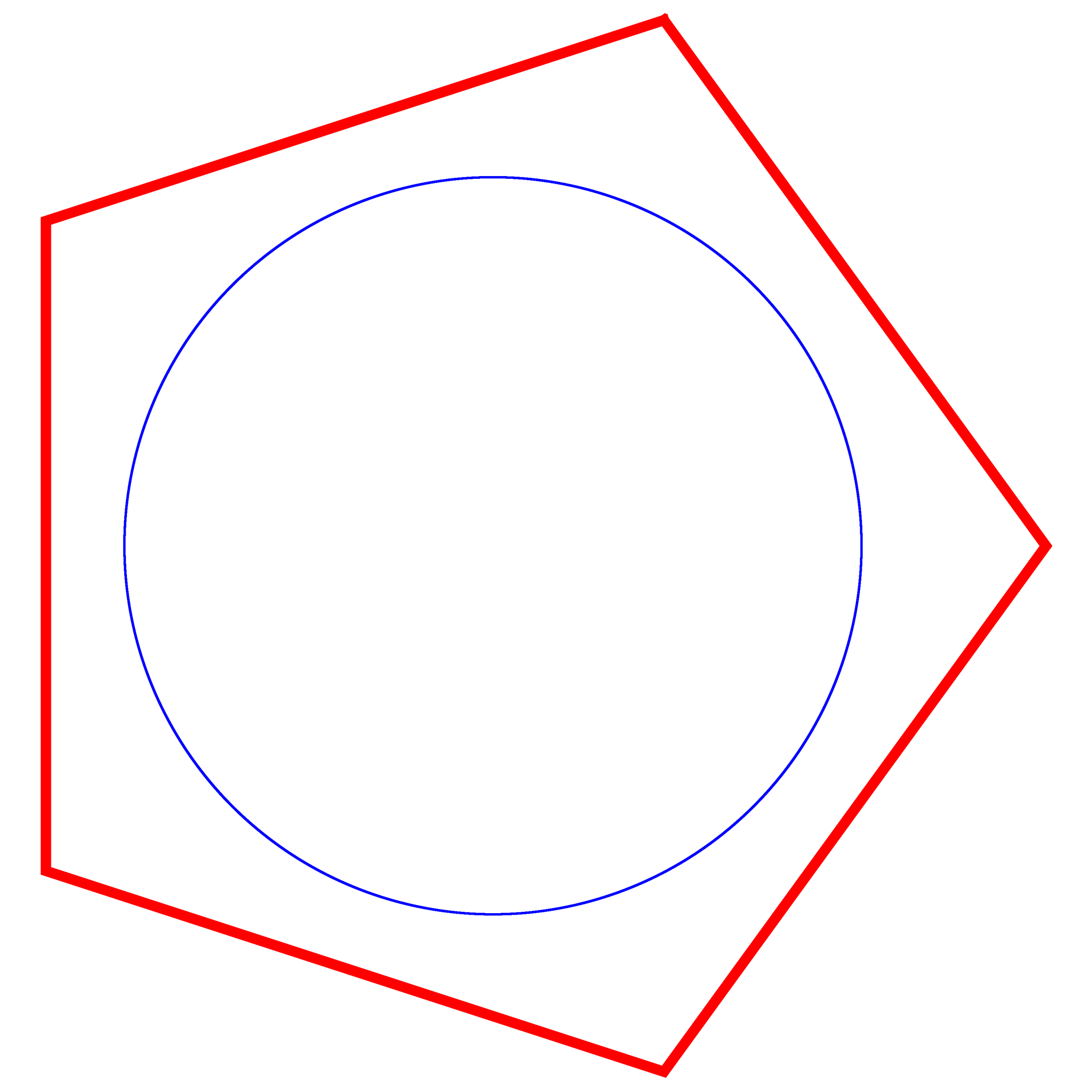}}
\scalebox{0.1}{\includegraphics{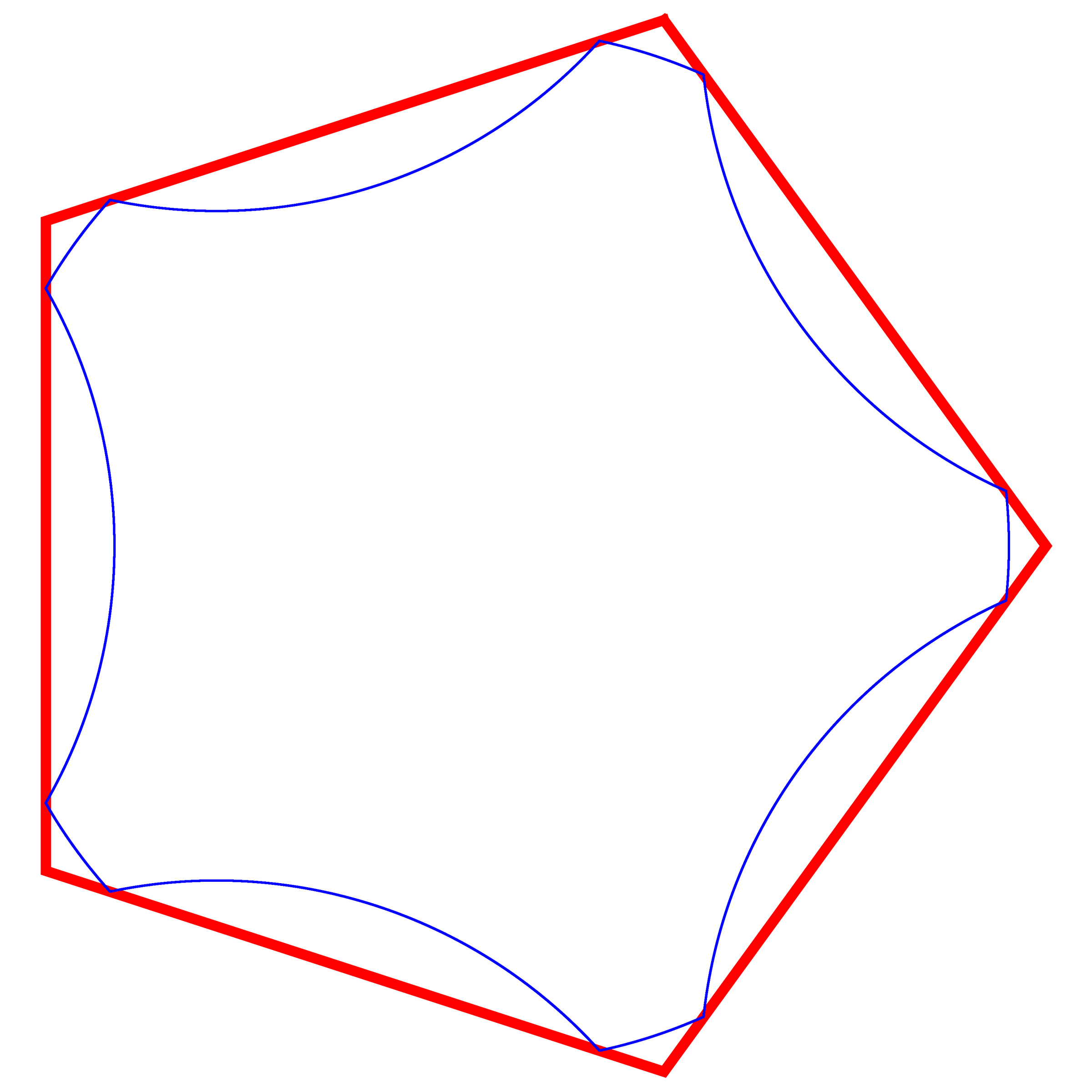}}
\scalebox{0.1}{\includegraphics{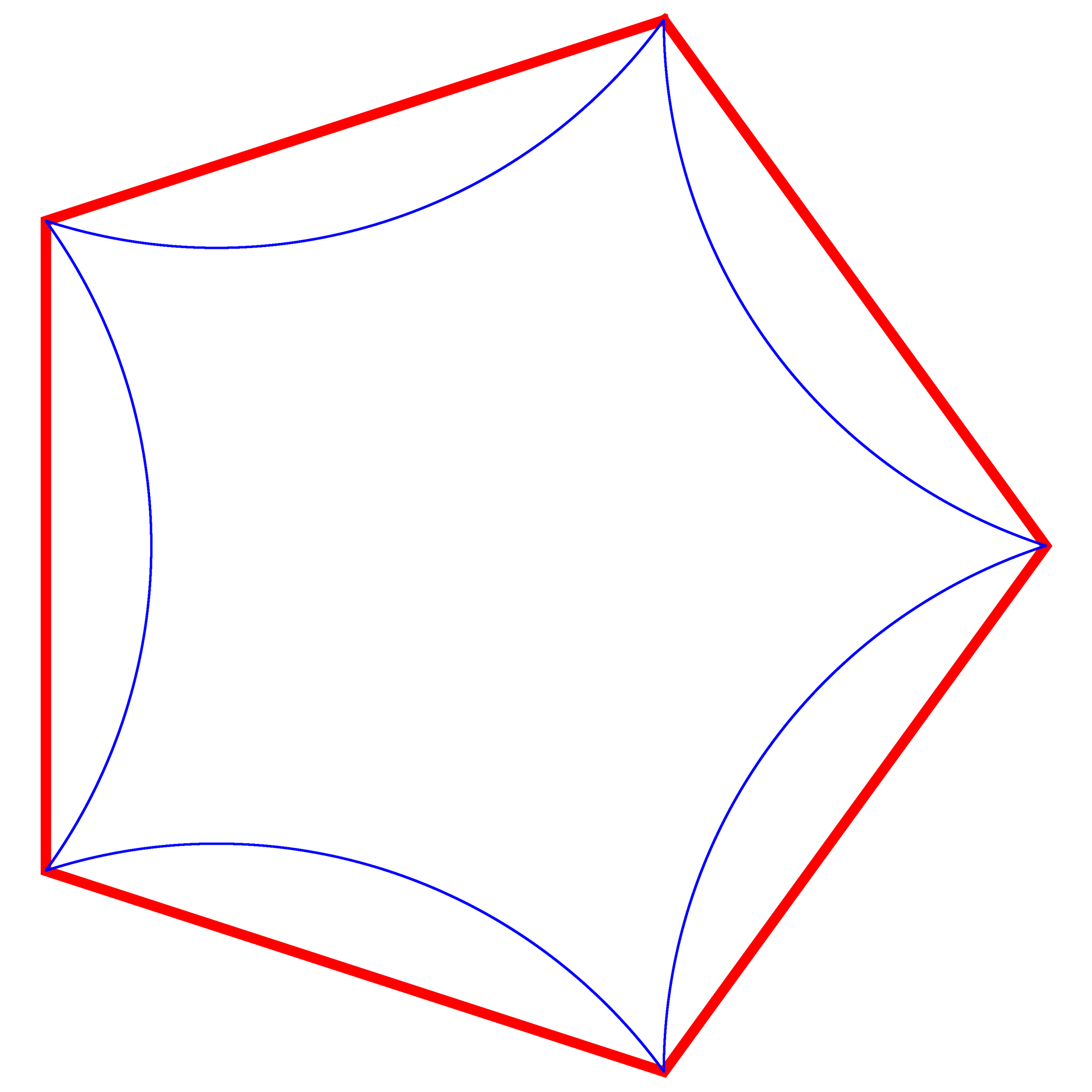}}
\scalebox{0.1}{\includegraphics{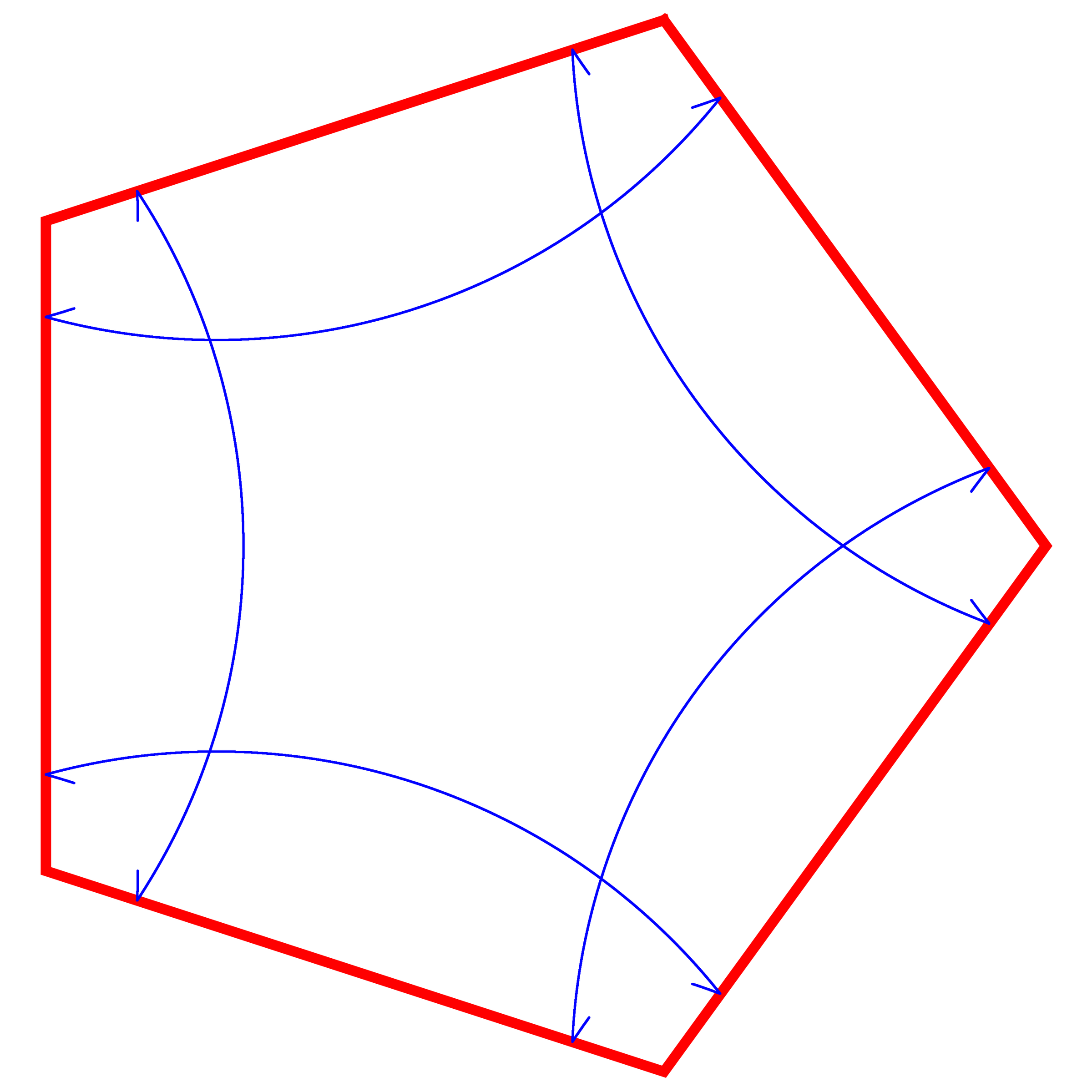}}
\scalebox{0.1}{\includegraphics{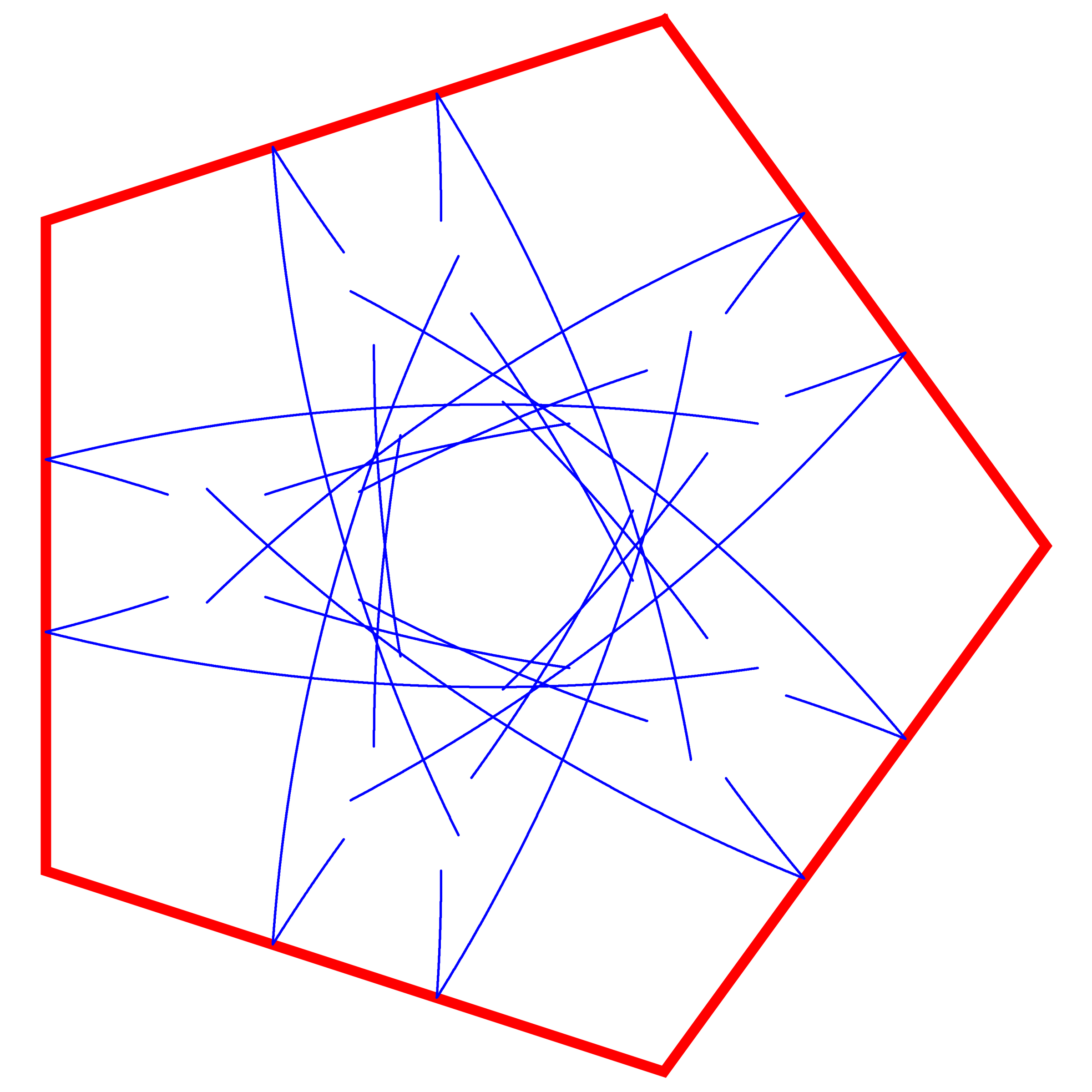}}
\scalebox{0.1}{\includegraphics{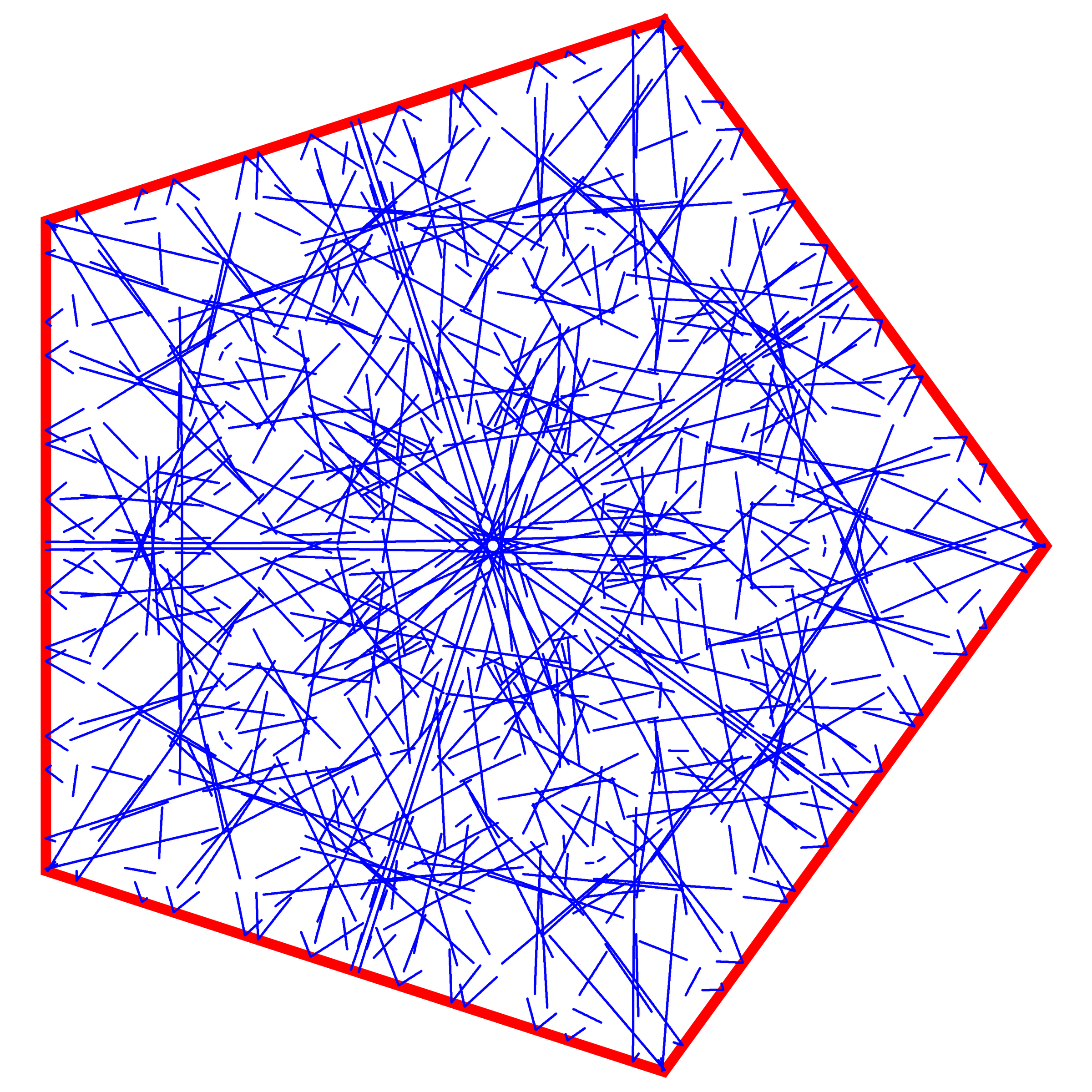}}
\label{Pentagon}
\caption{
We see wave fronts starting at the origin in a regular pentagon of side length $1$. 
The radii of the fronts are $r=0.66, 0.93, 1, 1.16, 3.33, 6.66, 20$. At $r=1$, when $W_t(P)$ reaches
a vertex for the first time, this cuts the wave front. For these pictures, we evolved 
100'000 billiard trajectories.
}
\end{figure}  

\paragraph{}
Lets look now at the situation with a non-flat smooth metric 
$$ g=\left[ \begin{array}{cc} a^2 \cos^2(u)+1 & -a \cos(u) \cos(v) \\ -a \cos(u) \cos(v) & \cos^2(v)+1 \\ \end{array} \right] $$
on the torus $M=\mathbb{T}^2=\mathbb{R}^2/(2\pi \mathbb{Z})^2$. The curvature is 
$-a \sin(u) \sin(v)/(1+a^2 \cos^2(u) + \cos^2(v))^2$.  The pictures show the case with $a=1.2$. 
We then integrate numerically the {\bf geodesic differential equations}
$\ddot{x}^k = - \sum_{i,j} \Gamma_{ij}^k \dot{x}^i \dot{x}^j$ 
with initial condition $x(0)=p, x'(0)=v$ for 100'000 directions.

\begin{figure}[!htpb]
\scalebox{0.07}{\includegraphics{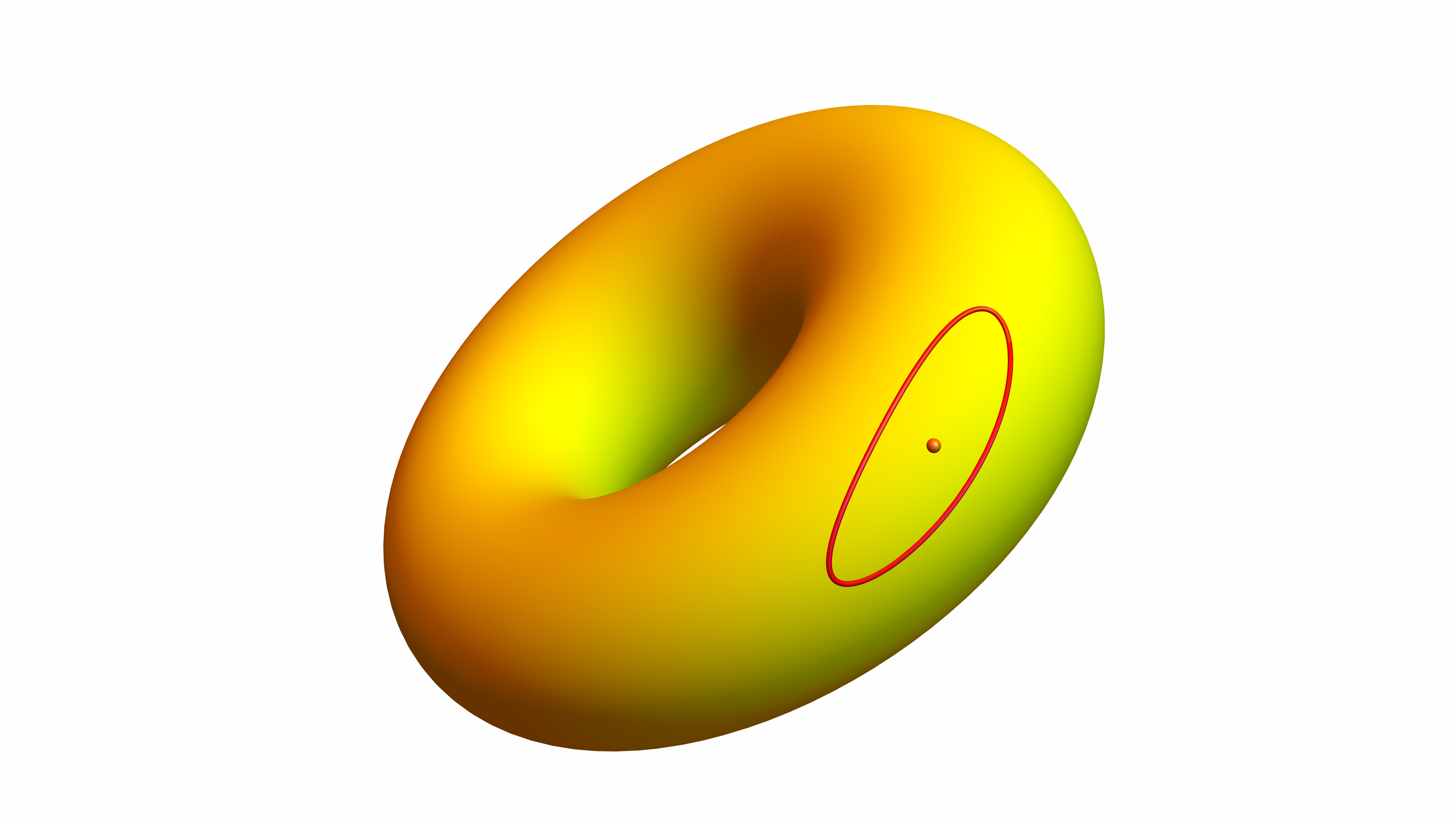}}
\scalebox{0.07}{\includegraphics{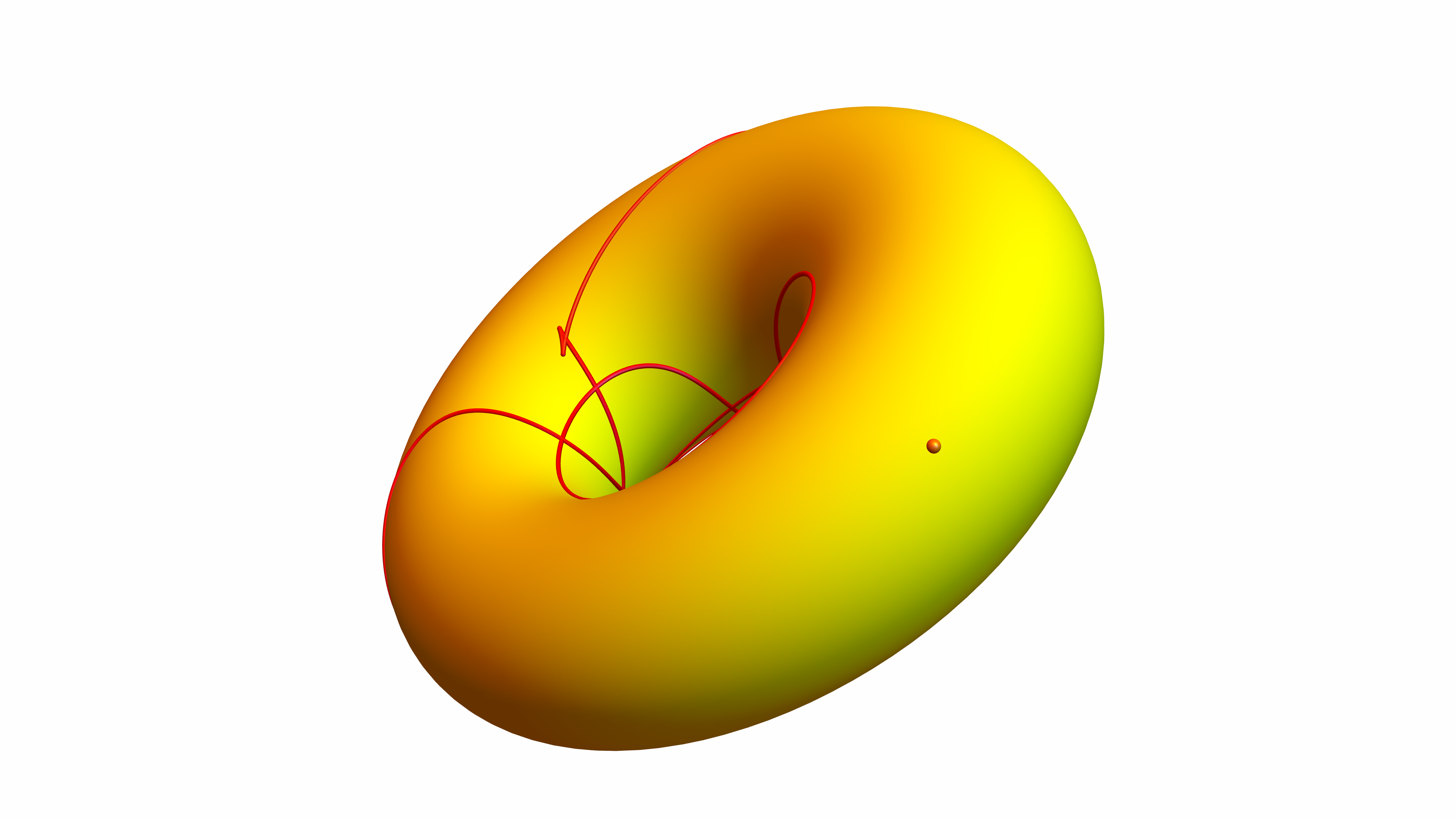}}
\scalebox{0.07}{\includegraphics{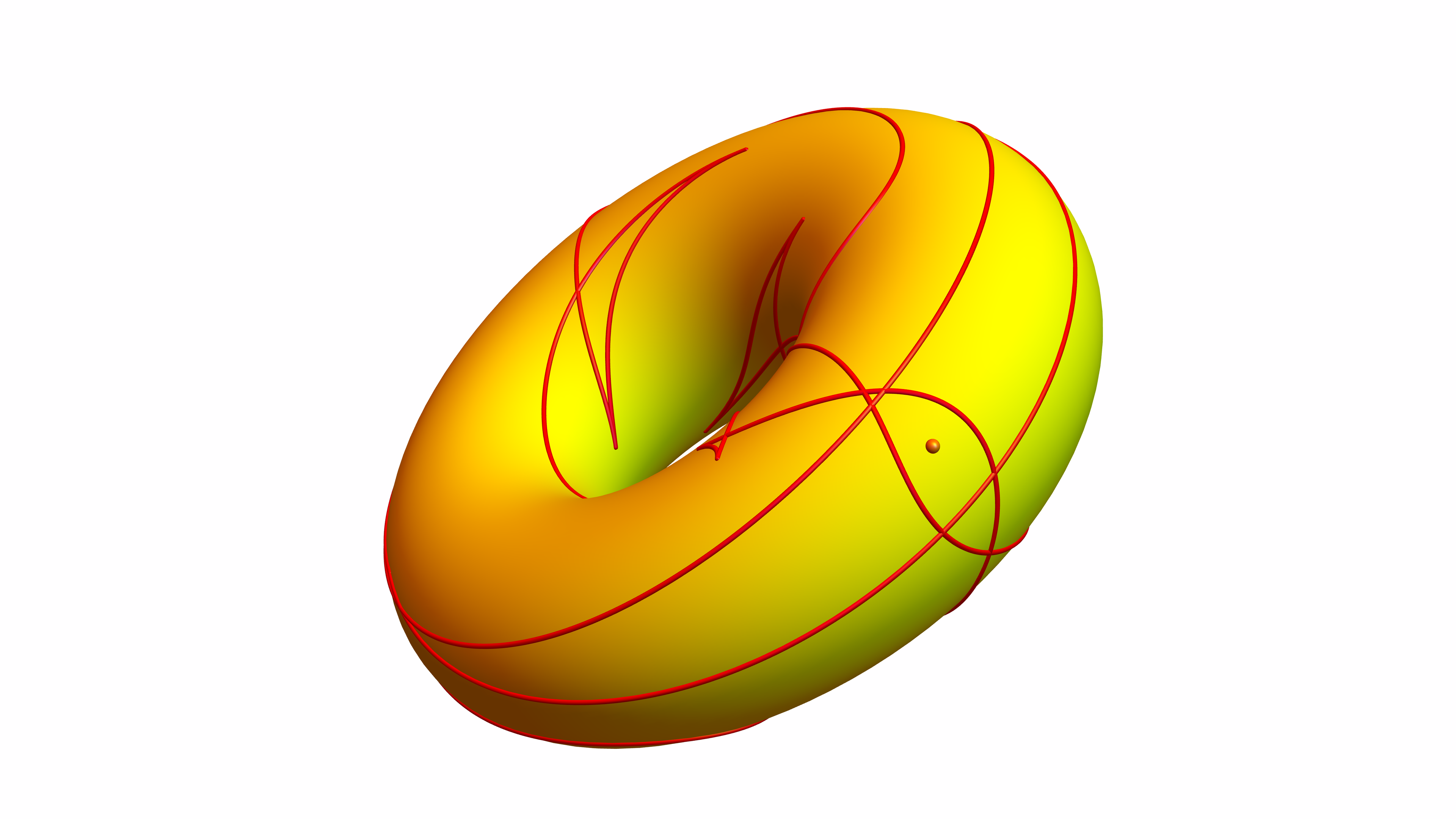}}
\scalebox{0.07}{\includegraphics{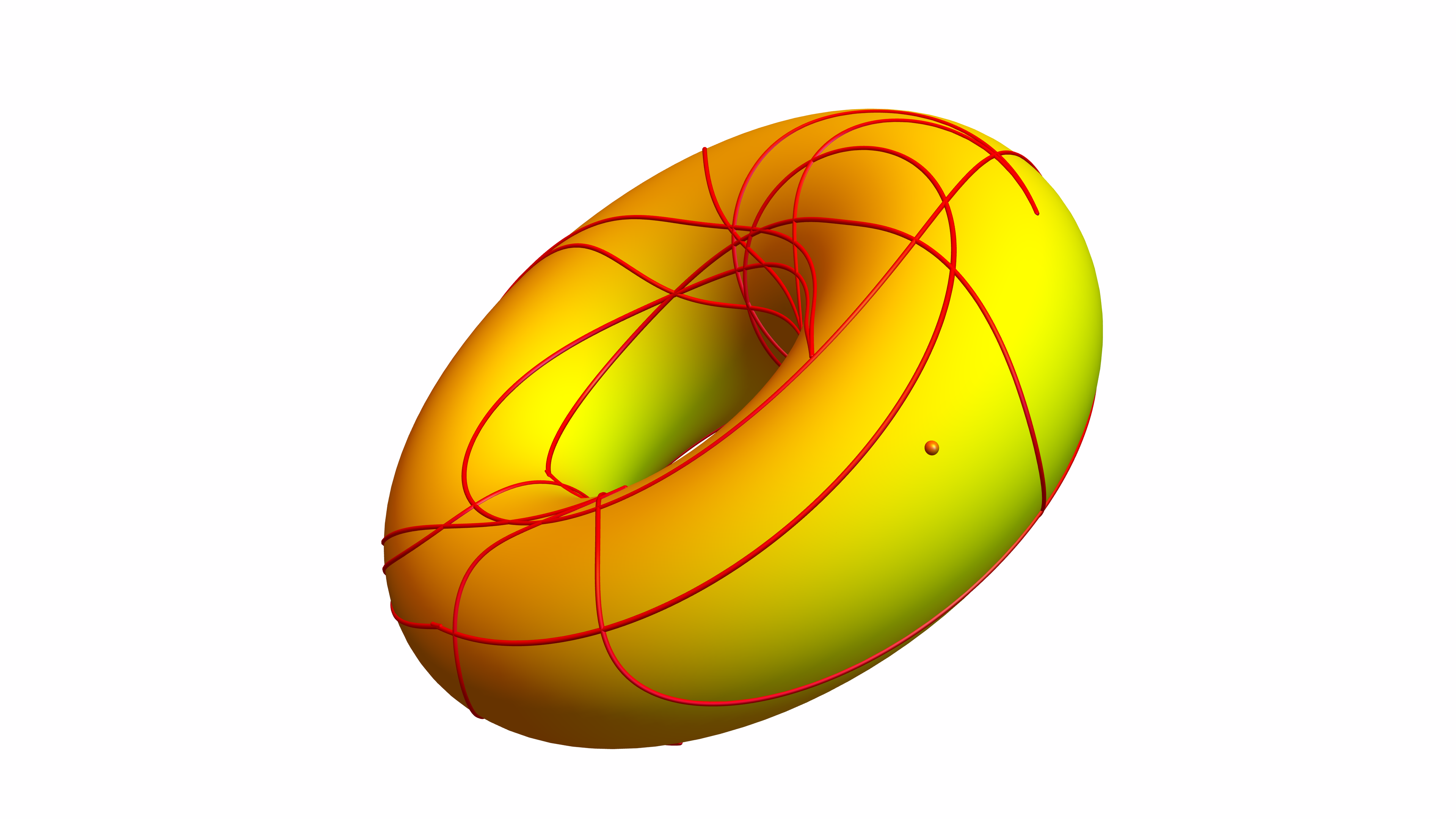}}
\scalebox{0.07}{\includegraphics{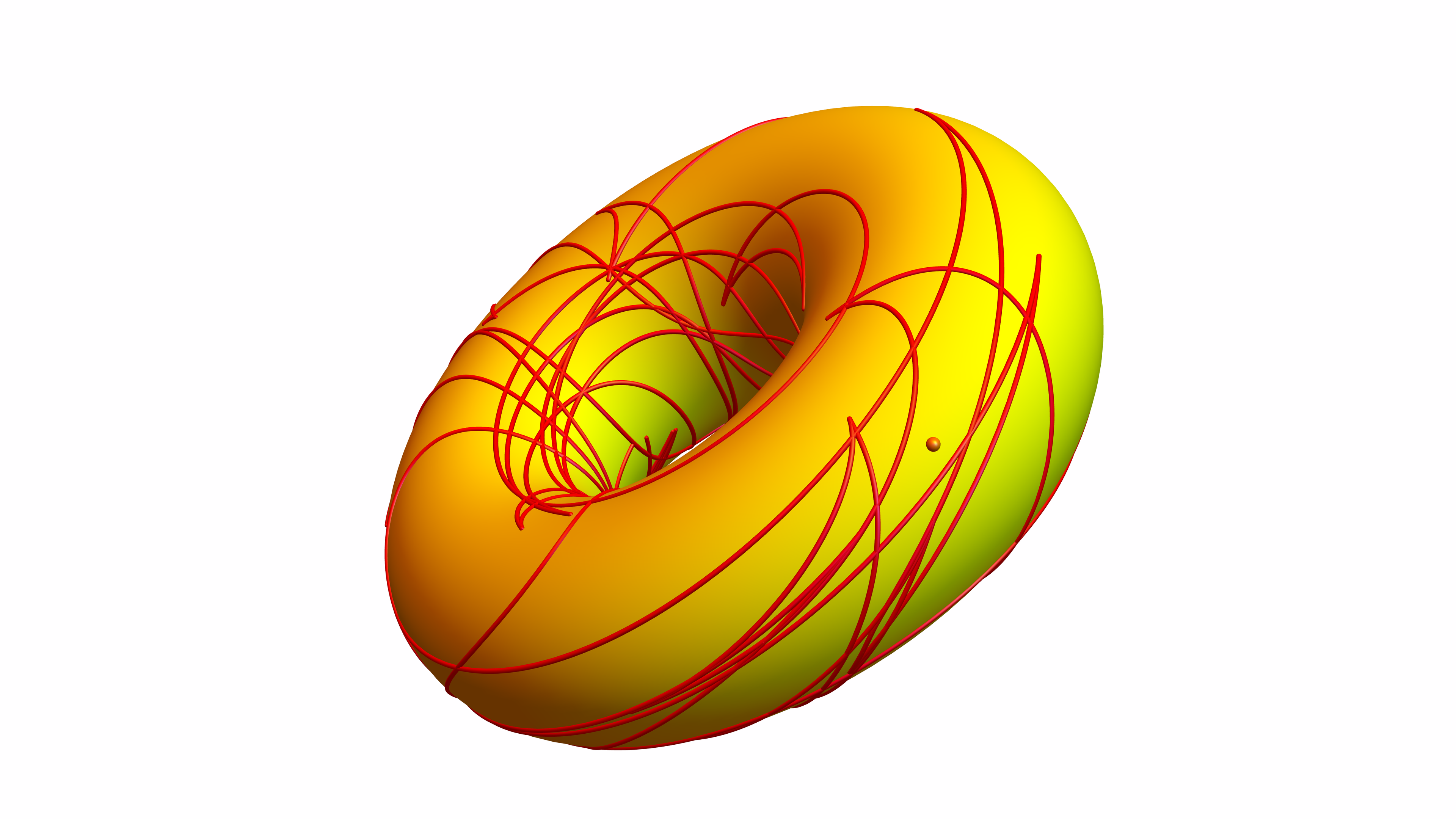}}
\scalebox{0.07}{\includegraphics{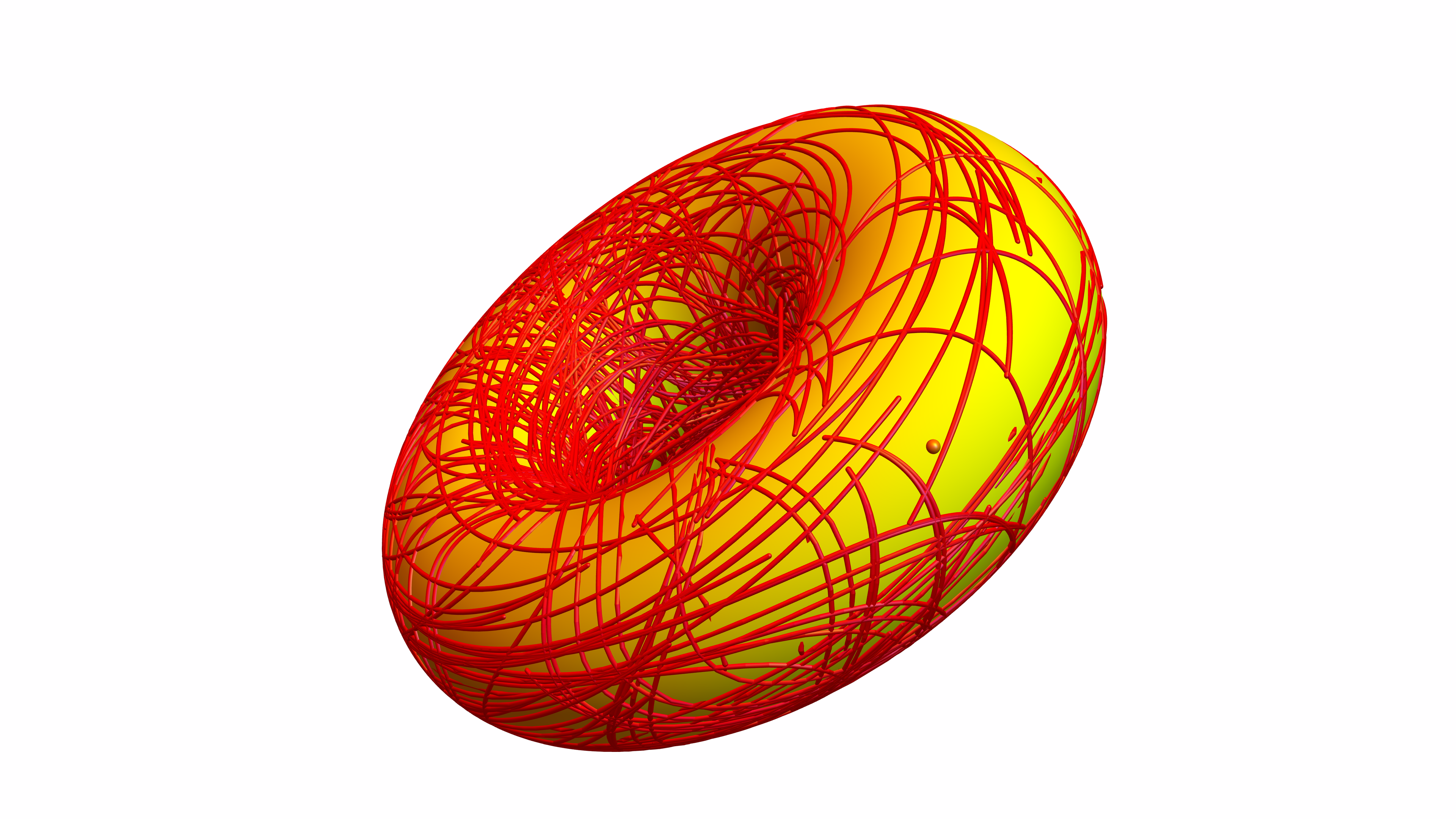}}
\label{Pentagon}
\caption{
Wave fronts on a non-flat torus $(M,g)$, where $g$ is a smooth Riemannian metric. 
The geodesic differential equations define a flow on the unit tangent bundle of 
$N=M \times \mathbb{T}^1$. This system preserves the volume and is in general expected to show
sensitive dependence on initial conditions on some part of $N$. In such a
case, we would expect the wave front length to grow exponentially.
}
\end{figure}  

\section*{Method}

\paragraph{}
All illustrations and experiments were done with the computer algebra
system Mathematica. Computer code for the illustrations is available
online on the Wolfram blog \cite{DensityWaveFrontsMathematica}.

\section*{Declarations}

\paragraph{}
The authors did not receive support from any organization for the submitted work.
No funding was received to assist with the preparation of this manuscript.
No funding was received for conducting this study.
No funds, grants, or other support was received.

\paragraph{}
All authors whose names appear on the submission
made substantial contributions to the conception
or design of the work.

\bibliographystyle{plain}

\begin{thebibliography}{10}

\bibitem{AMR}
R.~Abraham, J.E. Marsden, and T.~Ratiu.
\newblock {\em Manifolds, Tensor Analysis and Applications}.
\newblock Applied Mathematical Sciences, 75. Springer Verlag, New York etc.,
  second edition, 1988.

\bibitem{AlexanderBergBishop}
S.B. Alexander, D.I. Berg, and R.~Bishop.
\newblock Cut loci, minimizers, and wavefronts in riemannian manifolds with
  boundary.
\newblock {\em Michigan Math. J.}, 40:229--237, 1993.

\bibitem{BergerPanorama}
M.~Berger.
\newblock {\em A Panoramic View of Riemannian Geometry}.
\newblock Springer, 2003.

\bibitem{BergerGostiaux}
M.~Berger and B.~Gostiaux.
\newblock {\em Differential geometry: manifolds, curves, and surfaces}, volume
  115 of {\em Graduate Texts in Mathematics}.
\newblock Springer-Verlag, New York, 1988.

\bibitem{berglund2021uncentered}
N.~Berglund.
\newblock Drop in a circular water bowl, 2021.
\newblock https://www.youtube.com/watch?v=17DHYxcfpuk, Accessed: 2024-08-10.

\bibitem{berglund2021square}
N.~Berglund.
\newblock Squaring the circle, 2021.
\newblock https://www.youtube.com/watch?v=XCqQNJcrEdw, Accessed: 2024-08-10.

\bibitem{berglund2021pentagon}
N.~Berglund.
\newblock What's going on with the pentagon?, 2021.
\newblock https://www.youtube.com/watch?v=KCm0o2vYrD4, Accessed: 2024-08-10.

\bibitem{Besse}
A.~Besse.
\newblock {\em Manifolds all of whose geodesics are closed}, volume~93 of {\em
  Ergebnisse der Mathematik und ihrer Grenzgebiete}.
\newblock Springer Verlag, Berlin, 1978.

\bibitem{bialy2022birkhoff}
M.~Bialy and A.E. Mironov.
\newblock The {B}irkhoff-{P}oritsky conjecture for centrally-symmetric billiard
  tables.
\newblock {\em Ann. of Math. (2)}, 196(1):389--413, 2022.
\newblock Birkhoff-Poritsky conjecture, caustics, and low regularity
  conditions, and partial spectrum information.

\bibitem{Birkhoff}
G.D. Birkhoff.
\newblock {\em Dynamical systems}.
\newblock Colloquium Publications, Vol. IX. American Mathematical Society,
  Providence, R.I., 1966.

\bibitem{blaschke}
W.~Blaschke.
\newblock {\em Vorlesungen {\"uber} Integralgeometrie}.
\newblock Chelsea Publishing Company, New York, 1949.

\bibitem{Bun79}
L.A. Bunimovich.
\newblock On the ergodic properties of nowhere dispersing billiards.
\newblock {\em Communications in Mathematical Physics}, 65:295--312, 1979.

\bibitem{cardinzanelli2017schrodinger}
F.~Cardin and L.~Zanelli.
\newblock The geometry of the semiclassical wave front set for {Schr\"odinger}
  eigenfunctions on the torus.
\newblock {\em Math. Phys. Anal. Geom.}, 20(2):10--20, 2017.

\bibitem{ChernovMarkarian}
N.~Chernof and R.~Markarian.
\newblock {\em Chaotic billiards}.
\newblock AMS, 2006.

\bibitem{DavisLelievre}
{D. Davis and S. Leli\`evre}.
\newblock Periodic paths on the pentagon, double pentagon and golden l.
\newblock https://arxiv.org/pdf/1810.11310, 2024.

\bibitem{Davis2016}
D.~Davis.
\newblock {\em Lines in positive genus}, pages 1--54.
\newblock EMS, 2016.

\bibitem{VerdiereVicente2020}
Y.C. de~Verdi\`ere and D.~Vicente.
\newblock Large-time asymptotics of the wave fronts length {I}, the {E}uclidean
  disk.
\newblock 2020.

\bibitem{VerdiereVicente2021}
Y.C. de~Verdi\`ere and D.~Vicente.
\newblock Large-time asymptotics of the wave fronts length {II}, surfaces with
  integrable {H}amiltonians.
\newblock {\em Hal open science}, 2021.
\newblock HAL Id: hal-03097610.

\bibitem{DehnHeegaard}
M.~Dehn and P.~Heegaard.
\newblock Analysis situs.
\newblock {\em Enzyklopaedie d. Math. Wiss}, III.1.1:153--220, 1907.

\bibitem{DoCarmo1992}
M.P. do~Carmo.
\newblock {\em Riemannian Geometry}.
\newblock {Birkh\"auser}, 1992.

\bibitem{bor2024cusps}
M.~Spivakovsky G.Bor and S.Tabachnikov.
\newblock Cusps of caustics by reflection in ellipses.
\newblock {\em J. Lond. Math. Soc.}, 110(6), 2024.

\bibitem{gruenbaum}
B.~Gr\"unbaum.
\newblock {\em Convex Polytopes}.
\newblock Springer, 2003.

\bibitem{EberhardHopf1932}
E.~Hopf.
\newblock Complete transitivity and the ergodic principle.
\newblock {\em Proc. Nat. Acad. Sci.}, 18:204--209, 1932.

\bibitem{huxley1996primitive}
M.N. Huxley and W.G. Nowak.
\newblock Primitive lattice points in convex planar domains.
\newblock {\em Acta Arithmetica}, 76(3):271--283, 1996.

\bibitem{jakobson1997laplacian}
D.~Jakobson.
\newblock Quantum limits on flat tori.
\newblock {\em Ann. of Math. (2)}, 145(2):235--266, 1997.

\bibitem{KatokStrelcyn}
A.~Katok and J.-M. Strelcyn.
\newblock {\em Invariant manifolds, entropy and billiards, smooth maps with
  singularities}, volume 1222 of {\em Lecture notes in mathematics}.
\newblock Springer-Verlag, 1986.

\bibitem{Klingenberg1978}
W.~Klingenberg.
\newblock {\em Lectures on closed geodesics}, volume 230 of {\em Grundlehren
  der mathematischen Wissenschaften}.
\newblock Springer, 1978.

\bibitem{DensityWaveFrontsMathematica}
O.~Knill.
\newblock {Density of wave fronts: flat tori proof extends to cube, square and
  Klein}, 2025.
\newblock https://community.wolfram.com/groups/-/m/t/3442267.

\bibitem{CFZ}
O.~Knill and E.~Slavkovsky.
\newblock Visualizing mathematics using 3d printers.
\newblock In C.~Fonda E.~Canessa and M.~Zennaro, editors, {\em Low-Cost 3D
  Printing for science, education and Sustainable Development}. ICTP, 2013.

\bibitem{lakatos}
I.~Lakatos.
\newblock {\em Proofs and Refutations}.
\newblock Cambridge University Press, 1976.

\bibitem{lax2000geometry}
P.S. Lax.
\newblock Geometry and the number of lattice points.
\newblock {\em Proceedings of Symposia in Pure Mathematics}, 69:219--237, 2000.

\bibitem{LelievreMonteilWeiss2016}
S.~Leli\`evre, T.~Monteil, and B.~Weiss.
\newblock Everything is illuminated.
\newblock {\em Geometry \& Topology}, 20:1737--1762, 2016.

\bibitem{Por50}
H.~Poritsky.
\newblock The billiard ball problem on a table with a convex boundary-an
  illustrative dynamical problem.
\newblock {\em Annals of Mathematics}, 51:456--470, 1950.

\bibitem{Richeson}
D.S. Richeson.
\newblock {\em Euler's Gem}.
\newblock Princeton University Press, Princeton, NJ, 2008.
\newblock The polyhedron formula and the birth of topology.

\bibitem{Tab95}
S.~Tabachnikov.
\newblock {\em Billiards}.
\newblock Panoramas et synth\`eses. Soci\'et\'e Math\'ematique de France, 1995.

\bibitem{tokarsky1995polygonal}
G.~W. Tokarsky.
\newblock Polygonal rooms not illuminable from every point.
\newblock {\em The American Mathematical Monthly}, 102(10):867--879, 1995.

\bibitem{Vicente2020}
D.~Vicente.
\newblock Une goutte d'eau dans un bol.
\newblock {\em Quadrature}, 117:13--22, 2020.

\bibitem{KozlovTreshchev}
D.V.~Treshchev V.V.~Kozlov.
\newblock {\em Billiards}, volume~89 of {\em Translations of mathematical
  monographs}.
\newblock AMS, 1991.

\bibitem{Kuehnel2015}
{W. K\"uhnel}.
\newblock {\em Differential Geometry: Curves - Surfaces- Manifolds}, volume~77
  of {\em Student Mathematical Library}.
\newblock AMS, third edition, 2015.

\bibitem{waxman2023number}
E.~Waxman and N.~Yesha.
\newblock On the number of lattice points in thin sectors.
\newblock {\em arXiv preprint}, 2023.

\bibitem{wolecki2024illumination}
A.~Wolecki.
\newblock Illumination in rational billiards, 2024.
\newblock https://arxiv.org/abs/1905.09358.

\bibitem{Ziegler}
G.M. Ziegler.
\newblock {\em Lectures on Polytopes}.
\newblock Springer Verlag, 1995.

\end{thebibliography}

\end{document}